\documentclass[10pt,a4paper,reqno]{amsart}
\usepackage[a4paper,margin=2.5cm]{geometry}
\usepackage[english]{babel}
\usepackage[colorlinks,citecolor=blue,urlcolor=blue]{hyperref}
\usepackage{amssymb,bbm}
\usepackage{mathdots,mathtools}
\usepackage{xcolor,graphicx}
\numberwithin{equation}{section}
\allowdisplaybreaks
\newcommand{\even}{\mbox{{\rm\tiny even}}}
\newcommand{\odd}{\mbox{{\rm\tiny odd}}}
\usepackage{amsthm}
\theoremstyle{plain}% default
\newtheorem{thm}{Theorem}[section]
\newtheorem{lem}[thm]{Lemma}
\newtheorem{pro}[thm]{Proposition}
\newtheorem{cor}[thm]{Corollary}
\theoremstyle{remark}
\newtheorem{rmk}{Remark}[section]
\title[Fleming\,--\,Viot on cycle graph]{Dynamics of a Fleming\,--\,Viot type particle system\\on the cycle graph}
\author{Josu\'e Corujo}
\address[1]{CEREMADE, Universit\'e Paris-Dauphine, Universit\'e PSL, CNRS, 75016 Paris, France}
\address[2]{Institut de Math\'ematiques de Toulouse, Universit\'e de Toulouse, Institut National des Sciences Appliqu\'ees, 31077 Toulouse, France
} 
\email{jcorujo@insa-toulouse.fr}
\date{January 2020}

\begin{document}

\begin{abstract}
	We study the Fleming\,--\,Viot particle process formed by $N$ interacting continuous-time asymmetric random walks on the cycle graph, with uniform killing. 
	We show that this model has a remarkable exact solvability, despite the fact that it is non-reversible with non-explicit invariant distribution. 
	Our main results include quantitative propagation of chaos and exponential ergodicity with explicit constants, as well as formulas for covariances at equilibrium in terms of the Chebyshev polynomials. 
	We also obtain a bound uniform in time for the convergence of the proportion of particles in each state when the number of particles goes to infinity.
\end{abstract}

\keywords{Quasi-stationary distribution;  Fleming\,--\,Viot type particle system; Moran type model; propagation of chaos; ergodicity}

\subjclass[2000]{60K35, 60B10, 37A25}

\maketitle

% {\footnotesize \tableofcontents}

\section{Introduction}

This paper deals with a continuous-time Markov process describing the position of $N$ particles moving around on the cycle graph.
This type of model is usually known as Fleming\,--\,Viot process, or Moran type process
\cite{cloez2016quantitative,Etheridge_2009,ferrari2007countable}.
Consider a continuous-time Markov process on $E \cup \{\partial\}$, where $E$ is finite and $\partial$ is an absorbing state.
Briefly, the Fleming\,--\,Viot process consists in $N$ particles moving in $E$ as independent copies of the original process, until one of the particles gets absorbed. When this happens, the absorbed particle jumps instantaneously and
uniformly to one of the positions of the other particles.
The Fleming\,--\,Viot processes were originally and independently introduced  by Del Moral,
Guionnet, Miclo \cite{DelMoral1999, Moral2000}
and Burdzy, Ho{\l}yst, March \cite{Burdzy2000} to approximate the law of a Markov process conditioned to non-absorption,
and its \emph{Quasi-Stationary Distribution} (QSD), which is the limit
of this conditional law when $t \rightarrow \infty$. See e.g.\ the works of M\'el\'eard and Villemonais \cite{meleard2012quasi}, Collet et al.\
\cite{collet2012quasi} and van Doorn et al.\ \cite{van2013quasi}, excellent references for an
introduction to the theory related to the QSD. For recent and quite general results about the convergence of Markov processes conditioned to non-absorption to a QSD, we refer the interested reader to \cite{Champagnat2016}, \cite{ChampagnatVillemonaisP2017a} and \cite{bansaye:hal-02062882}.

The convergence of the empirical distributions induced by Fleming\,--\,Viot processes defined on discrete
state spaces when the size of the population and the time increase have been assured under some assumptions.
For example, Ferrari and Mari\'c \cite{ferrari2007countable} and Asselah et al.\ \cite{asselah2011finite} study
the convergence of the empirical distribution induced by the
Fleming\,--\,Viot process to the unique QSD in countable and
finite discrete space settings, respectively. With the aim to study the convergence of the
particle process under the stationary distribution to the QSD, Leli\`evre et al.\ \cite{CLT_Alea} 
proves a Central Limit Theorem for the finite state case. Additionally,
Villemonais \cite{villemonais2015} and Asselah et al.\ \cite{asselah2016Galton-Watson} study the convergence to the minimal QSD in a
Galton\,--\,Watson type model and in a birth and death process, respectively. Similarly, Asselah and Thai \cite{asselah2012note} and Mari\`c \cite{Maric2015} address the study of
the $N$-particle system associated to a random walk on $\mathbb{N}$ with a
drift towards the origin, which is an absorbing state. In these scenarios
there exist infinitely many QSD for each model, so it is important to ensure
the ergodicity of the $N$-particle system and to determine to which QSD it
converges. Additionally, Champagnat and Villemonais \cite{ChampagnatVillemonaisP2018c} study the convergence of the Fleming\,--\,Viot process to the minimal QSD under general conditions, providing also some specific examples.  

In addition, some works have been devoted to the study of the speed of convergence when the number of particles and time tend to infinity. 
In particular, Cloez and Thai \cite{cloez2016quantitative} study the $N$-particle system in a discrete state space setting. They study the convergence of the empirical measure induced by the Fleming\,--\,Viot process when both $t \rightarrow \infty$ (ergodicity) and $N \rightarrow \infty$ (propagation of chaos), providing explicit bounds for the speed of convergence. 
Following the results in \cite{cloez2016quantitative}, Cloez and Thai \cite{cloez2016fleming} study two examples in details: the random walk on the complete graph with uniform killing and the random walk on the two-site graph. The simple geometries of the graphs of these models simplify the study of the $N$-particle dynamic and allows them to give explicit expressions for the stationary distributions of the $N$-particle processes and explicit bounds for its convergence to the QSD.

Consider the quantity $\lambda$ defined in \cite{cloez2016quantitative} as
\begin{equation}\label{def:lambda_Cloez}
	\lambda = \inf_{x, y} \Big( Q_{x,y} + Q_{y,x} + \sum_{s \neq x,y} Q_{x,s} \wedge Q_{y,s} \Big),
\end{equation}
where $Q= \big(Q_{x,y}\big)_{x,y}$ is the infinitesimal generator matrix of the process until absorption. 
When $\lambda = 0$ some of the results of \cite{cloez2016quantitative} do not hold and most of the bounds given become too rough. Note that $\lambda > 0$ for the two examples studied in \cite{cloez2016fleming}, but $\lambda$ is equal to zero for those models where  there exist two vertices such that the distance between them is greater than two. The quantity $\lambda$ is somehow related to the geometry of the graph associated to the Markov process. 
Hence, it becomes interesting to find explicit bounds for the speed of convergence of Fleming\,--\,Viot processes with more complex geometries.

In this article we focus on the random walk on the cycle graph $\mathbb{Z}/K \mathbb{Z}$ for $K \ge 3$. Note that for this graph it holds that $\lambda = 0$ when $K \ge 6$. For simplicity, we assume that the $N$ particles jump to the absorbing state with the same rate, i.e., we consider a process with \emph{uniform killing} (cf.\ \cite{meleard2012quasi}). Even if in this case the distribution of the conditional process is trivial, the study of the Fleming\,--\,Viot process becomes more complicated due to its non reversibility and the geometry of the cycle graph. We focus on providing bounds for the speed of the convergence of the empirical distribution induced by the particle system to the unique QSD when $t$ and $N$ tend to infinity. This example can be seen as a further step towards the study of the speed of convergence of Fleming\,--\,Viot process with more general geometry.

\subsection{Model and notations}

Consider a Markov process $(Z_t)_{t \ge 0}$ with state space $\mathbb{Z}/K\mathbb{Z}\, \cup \, \{ \partial \}$, where $K \ge 3$ and $\partial$ is an absorbing state. Specifically, the infinitesimal generator of the process is given by
\[
\mathcal{G} f(x) =
f(x + 1) - f(x) + \theta [ f(x-1) - f(x) ] + p[f (\partial) - f(x)],
\]
where $x \in \mathbb{Z}/K\mathbb{Z}$, $\mathcal{G}f(\partial) = 0$, $\theta, p \in \mathbb{R_+^*}$ and $f$ is a real function defined on $\mathbb{Z}/K\mathbb{Z} \, \cup \, \{ \partial \}$. In words, $(Z_t)_{t \ge 0}$ is an asymmetric random walk on the $K$-cycle graph, which jumps with rates $1$ and $\theta$ in the clockwise and the anti-clockwise directions, respectively. Also, with uniform rate $p$ the process jumps to the absorbing state $\partial$, i.e., it is killed.
Note that $\mathbb{Z}/K\mathbb{Z}$ is an irreducible class.
The process generated by $\mathcal{G}$ is a particular case of the processes with \emph{uniform killing in a finite state space} considered by M\'el\'eard and Villemonais \cite[\S~2.3]{meleard2012quasi}.

Let $(X_t)_{t \ge 0}$ be the analogous asymmetric random walk on the cycle graph $\mathbb{Z}/K\mathbb{Z}$ without killing. The generator of this process, denoted by $\mathcal{H}$, is given by
\[
\mathcal{H} f(x) = f(x + 1) - f(x) + \theta [ f(x-1) - f(x) ], \; \text{ for all } \; x \in \mathbb{Z}/K\mathbb{Z}.
\]
Note that, because of the uniform killing, the process $(Z_t)_{t \ge 0}$ could also be defined in the following way
\[
Z_t = \left\{
\begin{array}{ccc}
	X_t & \text{ if } & t < \tau_p\\
	\partial & \text{ if } & t \ge \tau_p,
\end{array}
\right.
\]
where $\tau_p$ is an exponential random variable with mean $1/p$ and independent of the random walk $(X_t)_{t \ge 0}$.
This means that the law of the process $(Z_t)_{t \ge 0}$ conditioned to non-absorption is given by
\[
\mathbb{P}_{\mu}[Z_t = k \,|\, t < \tau_p] = \mathbb{P}_{\mu}[X_t = k],
\]
for $k \in \mathbb{Z}/K\mathbb{Z}$ and for every initial distribution $\mu$ on $\mathbb{Z}/K\mathbb{Z}$.
As a consequence, the QSD of $(Z_t)_{t\ge 0}$, denoted by $\nu_{\mathrm{qs}}$, is the stationary distribution of $(X_t)_{t \ge 0}$, which is the uniform distribution on $\mathbb{Z}/K\mathbb{Z}$, as we will prove in Theorem \ref{thm1}.

Recall that the total variation norm of a signed measure $\mu$ defined on a
discrete probability space $E$ is given by
$\|\mu\|_{\mathrm{TV}} = \frac{1}{2} \|\mu\|_1$ where
$\|\mu\|_p = (\sum_{x \in E} |\mu(x)|^p)^{1/p}$ is the $p$-norm, see for
instance \cite[\S~4.1]{peres2017markov}. If $(f_N)$ and $(g_N)$ are two real
sequences, $f_N \mathrel{\mathop{\sim}\limits_{N \rightarrow \infty}} g_N$ means
${f_N} - {g_N} = o \left(g_N\right)$.

Now, assume we have $N$ particles with independent behavior driven by the generator $\mathcal{G}$, until one of them jumps to the absorbing state. When this happens, the particle instantaneously and uniformly jumps to one of the positions of the other $N-1$ particles. 
We denote by $\big( \eta^{(N)}_t \big)_{t \ge 0}$ the Markov process, which accounts the positions of the $N$ particles in the $K$-cycle graph at time $t$. 
Consider the state space $\mathcal{E}_{K,N}$ of this process, which is given by
\begin{equation*}\label{def_space_state}
	\mathcal{E}_{K,N} = \left\{ \eta : \mathbb{Z}/K \mathbb{Z} \rightarrow \mathbb{N}, \sum\limits_{k = 0}^{K - 1} \eta(k) = N \right\}.
\end{equation*}

At time $t$ the system is in state $\eta_t = (\eta_t(0), \eta_t(1), \dots, \eta_t(K-1))$ if there are $\eta_t(k)$ particles on site $k$, for $k=0,1, \dots, K-1$. Note that the cardinality of $ \mathcal{E}_{K,N}$ is equal to that of the set of nonnegative solutions of the integer equation $x_1 + x_2 + \dots + x_K = N$, which is $ \mathrm{card}\,\big( \mathcal{E}_{K,N} \big) = \binom{K+N-1}{N}$, see e.g.\ \cite[Thm.\ D, \S~1.7]{comtet2012advanced}.

The generator of the $N$-particle process $\big(\eta^{(N)}_t\big)_{t \ge 0}$, denoted by $\mathcal{L}_{K,N}$, applied to a function $f$ on $\mathcal{E}_{K,N}$ reads
\begin{equation}\label{generator1}
	(\mathcal{L}_{K,N} f)(\eta) = \sum\limits_{i, j \in \mathbb{Z}/K \mathbb{Z}} \eta(i) \left( \mathbbm{1}_{\{ j = i + 1\}} + \theta \mathbbm{1}_{\{ j = i - 1\}} + p \frac{\eta(j)}{N-1} \right) [f( T_{i \rightarrow j} \eta ) - f(\eta) ],
\end{equation}
where $\theta, p >0$ and for every $\eta \in \mathcal{E}_{K,N}$ satisfying $\eta(i) > 0$, the configuration $T_{i \rightarrow j} \eta$ is defined as $T_{i \rightarrow j} \eta = \eta - \mathbf{e}_i + \mathbf{e}_j$ and $\mathbf{e}_i$ is the $i$-th canonical vector of $\mathbb{R}^K$. 
Under these dynamics, each of the $N$ particles, no matter where it is, can jump to every site $j \in \mathbb{Z}/K \mathbb{Z}$ such that $\eta(j) > 0$. Note that the process $\big(\eta_t^{(N)}\big)_{t \ge 0}$ is irreducible. Consequently, it has a unique stationary distribution denoted $\nu_N$.

For every $\eta \in \mathcal{E}_{K,N}$ the empirical distribution $m(\eta)$ associated to the configuration $\eta$ is defined by
\[
m( \eta ) = \frac{1}{N} \sum_{k=0}^{K-1} \eta(k) \delta_{ \{ k \} },
\]
where $\delta_{\{k\}}$ is the Dirac distribution at $k \in \mathbb{Z}/K \mathbb{Z}$.

The (random) empirical distribution $m\big(\eta_t^{(N)}\big)$ approximates the QSD of the process $(Z_t)_{t \ge 0}$ (cf.\ \cite{asselah2011finite,ferrari2007countable,Villemonais_2014}) which due to Theorem \ref{thm1} below is the uniform distribution. We are interested in studying how fast $m\big(\eta_t^{(N)}\big)$ converges to the uniform distribution on $\mathbb{Z}/K\mathbb{Z}$ when both $t$ and $N$ tend to infinity. Consider $\eta_{\infty}^{(N)}$ a random variable with distribution $\nu_N$, the stationary distribution of the process $\big(\eta_t^{(N)}\big)_{t \ge 0}$. In this work we develop a similar analysis to that of the complete graph dynamics in \cite{cloez2016fleming}. We focus on the convergences when both $N$ and $t$ tend to infinity, as shown in the following diagram
\[
\begin{array}{rcl}
	m\big(\eta_{t}^{(N)}\big) & \xrightarrow[t \rightarrow \infty]{}  & m\big(\eta_{\infty}^{(N)}\big) \\
	N  \Big\downarrow \hspace*{.4 cm}  &  &  \hspace*{.4 cm} \Big\downarrow N  \\
	\mathcal{L}(Z_{t} \mid t < \tau_p) & \xrightarrow[t \rightarrow \infty]{} & \hspace*{0.4 cm} \nu_{\mathrm{qs}}
\end{array}
\]
where the limits are in distribution. Theorem \ref{thm1} provides  lower and upper exponential bounds for the speed of convergence of $\mathcal{L}(Z_{t} \mid t < \tau_p)$ to $\nu_{\mathrm{qs}}$ in the $2$-norm, when $t \rightarrow \infty$. Likewise, Corollary \ref{corol:Convergence_QSD} and Theorem \ref{ergodicity_thm_2} give bounds for the speed of convergence of $m\big(\eta_{t}^{(N)}\big)$ to  $\mathcal{L}(Z_{t} \mid t < \tau_p)$ and $m\big(\eta_{\infty}^{(N)}\big)$ to $\nu_{\mathrm{qs}}$, when $N \rightarrow \infty$.

The quantitative long time behavior of the $N$-particle system in countable state spaces is studied in \cite{cloez2016quantitative}. Using a coupling technique and  under certain conditions, an exponential bound is provided for the convergence of $\mathcal{L} \big(\eta_{t}^{(N)}\big)$ to $\nu_N$ in the sense of a \emph{Wasserstein distance} \cite[Thm.\ 1.1]{cloez2016quantitative}. In particular, the parameter $\lambda$ defined by (\ref{def:lambda_Cloez}) needs to be positive. As we said, this is not the case of the asymmetric random walk on the $K$-cycle graph with uniform killing, when $K \ge 6$. 
A study of this convergence can be carried out using the spectrum of the generator $\mathcal{L}_{K,N}$, which is obtained in the recent paper \cite{2020arXiv201008809C}.
Indeed, using Example 3 in \cite{2020arXiv201008809C} we can get the following asymptotic expression for the profile of the convergence in total variation distance to stationarity:
\[
	\max_{\eta \in \mathcal{E}_{K,N}} \left\| \mathcal{L}_{\eta} \left( \eta_t^{(N)} \right)- \nu_{N} \right\|_{\mathrm{TV}} = \mathcal{O} \left( \mathrm{e}^{-\rho_K t} \right),
\]
where $\rho_K = 2(1 + \theta)\sin^2\left({\pi}/{K} \right)$,  $\mathcal{L}_{\eta} \left( \eta_t^{(N)} \right)$ stands for the law of the process generated by $\mathcal{L}_{K,N}$ at time $t$ and with initial distribution concentrated at $\eta \in \mathcal{E}_{K,N}$, and for a real positive function $f$ we denote by $\mathcal{O}(f)$ another real positive function such that
\[
	C_1 f(t) \le \mathcal{O}(f)(t) \le C_2f(t),
\]
for two constants $0 < C_1 \le C_2 < \infty$ and for all $t \ge T$, for $T > 0$ large enough.
It would be interesting to get non asymptotic results, with explicit constants, for the speed of convergence of the process generated by $\mathcal{L}_{K,N}$ to stationarity.
In order to do that, one possible alternative is to use the results in the recent paper of Villemonais \cite{Villemonais2019}, for a suitable distance, to get upper bounds for the speed of convergence in the sense of a \emph{Wasserstein distance}.
In addition, the recent work of Hermon and Salez \cite{MR3984254} offers clues to an alternative method for solving this problem: control the \emph{Dirichlet form} of the Fleming\,--\,Viot process in terms of the \emph{Dirichlet form} of a single particle. 
Moreover, it remains as an open question the study of the existence of a \emph{cutoff phenomenon} when the number of particles $N$ tends towards infinity.
These are possible directions for future research.

\subsection{Main results}

We first prove that the uniform distribution on $\mathbb{Z}/K \mathbb{Z}$ is the QSD of $(Z_t)_{t \ge 0}$. 
We also establish exponential bounds in the $2$-distance and the total variation distance between the distribution of this process at time $t$ and its QSD.

Let us denote by $\mathcal{L}_{\nu}(Z_t \mid t < \tau_p)$ the distribution at time $t$ of the asymmetric random walk on the cycle graph, $(Z_t)_{t \ge 0}$, with initial distribution $\nu$ on $\mathbb{Z}/K \mathbb{Z}$ and conditioned to non-absorption up to time $t$.
Let us denote by $\varphi_{\nu}$ the characteristic function of a distribution $\nu$ on $\mathbb{Z}/K \mathbb{Z}$, which satisfies
\[
  \varphi_{\nu}(t) = \mathbb{E}_{\nu}\left[\mathrm{e}^{\mathrm{i} t X}\right] = \sum_{k = 0}^{K-1} \nu(k) \mathrm{e}^{\mathrm{i} t k},
\]
for all $t \ge 0$ \cite[\S~3.3]{MR3930614}. Note that
\begin{equation*}
  \varphi_{\nu_{\mathrm{qs}}} (t) = \frac{1 - \mathrm{e}^{\mathrm{i} t K}}{K (1 - \mathrm{e}^{\mathrm{i} t})},
\end{equation*}
for all $t \ge 0$.
Let us denote by $\mathrm{D}_2(t)$ and $\mathrm{D}_{\mathrm{TV}}(t)$ the maximum distances to stationarity in the $2$-distance and in total variation at time $t$, respectively, which are defined as follows:
\begin{eqnarray*}
	\mathrm{D}_2(t) &=& \max_{\nu} \left\| \mathcal{L}_{\nu}(Z_t \mid  t < \tau_p) - \nu_{\mathrm{qs}} \right\|_{2},\\
	\mathrm{D}_{\mathrm{TV}}(t) &=& \max_{\nu} \left\| \mathcal{L}_{\nu}(Z_t \mid  t < \tau_p) - \nu_{\mathrm{qs}} \right\|_{\mathrm{TV}},
\end{eqnarray*}
where the maximum runs over all possible initial distributions $\nu$ on $\mathbb{Z}/K \mathbb{Z}$.
Since $\mathbb{Z}/K \mathbb{Z}$ is finite, we know that the convergence of $\mathcal{L}_{\nu}(Z_t \mid t < \tau_p)$ to $\nu_{\mathrm{qs}}$ is exponential \cite{Darroch_1967}. 
The following theorem gives exponential lower and upper bounds for this convergence.

\begin{thm}[Convergence in $2$-distance and total variation distance]\label{thm1}
  The QSD of the process $(Z_t)_{t \ge 0}$, $\nu_{\mathrm{qs}}$, is the
  uniform distribution on $\mathbb{Z}/K\mathbb{Z}$. Also, denoting
  \[
    \Delta_t(\mu,\nu)=\mathcal{L}_{\nu}(Z_t \mid t < \tau_p)-
    \mathcal{L}_{\mu}(Z_t \mid t < \tau_p),
  \]
  we have, for for every initial distributions $\nu$ and $\mu$ on
  $\mathbb{Z}/K\mathbb{Z}$ and every $t\geq0$,
  \begin{align}
    \left| \varphi_{\nu}\left( \frac{2 \pi}{K} \right) - \varphi_{\mu}\left( \frac{2 \pi}{K} \right) \right|  
    \mathrm{e}^{-\rho_K t} &\le \left\| \Delta_t(\mu,\nu) \right\|_{2}
    \le  \left\| \nu - \mu \right\|_2\mathrm{e}^{-\rho_K t},  \label{exp_bound_norm_2distrib} \\
    \frac{\sqrt{K}}{2} \left| \varphi_{\nu}\left( \frac{2 \pi}{K} \right) - \varphi_{\mu}\left( \frac{2 \pi}{K} \right) \right| 
    \mathrm{e}^{-\rho_K t} &\le 
    \left\| \Delta_t(\mu,\nu) \right\|_{\mathrm{TV}}  
    \le \frac{\sqrt{K}}{2} \left\| \nu - \mu \right\|_2 \; \mathrm{e}^{- \rho_K t},  \label{exp_bound_TVdistrib}
  \end{align}
  where
  \begin{eqnarray}
    \rho_K &=& 2(1 + \theta)\sin^2\left(\frac{\pi}{K} \right). \label{def:rhoK}
  \end{eqnarray}
  Moreover, the convergence of $\mathcal{L}_{\nu}(Z_t \mid t < \tau_p)$ to
  $\nu_{\mathrm{qs}}$ in the $2$-distance and the total variation distance is
  exponential with rate $-\rho_K$. Indeed, for all $t\geq0$,
  \begin{align}
    \frac{1}{\sqrt{K}} \mathrm{e}^{ - \rho_K t} \le 
    & \operatorname{D}_2(t) 
    \le 
    \sqrt{\frac{K - 1}{K}}  \mathrm{e}^{- \rho_K t},   \label{exp_bound_2norm} \\
    \frac{1}{2} \mathrm{e}^{- \rho_K t} \le 
    & \operatorname{D}_{\mathrm{TV}}(t) 
    \le \frac{1}{2}\sqrt{K- 1} \; \mathrm{e}^{- \rho_K t}. \label{exp_bound_TVnorm}
  \end{align}
\end{thm}

In spite of its simplicity, we did not find this result in the literature. Therefore, for the sake of completeness, we provide a proof of this theorem in Section \ref{sec:N_particles}.

Consider the function $\phi: \mathcal{E}_{K,N} \rightarrow \mathcal{E}_{K,N}$ defined by
\begin{equation}\label{def_rotation}
	\phi(\eta_0, \eta_1, \dots, \eta_{K-1}) = (\eta_1, \eta_2, \dots, \eta_{K-1}, \eta_0)
\end{equation}
and its $l$-composed $\phi^{(l)} = \phi \circ \phi \circ \dots \circ \phi$
($l$ times) which acts on the cycle graph by rotating it $l$ sites clockwise,
for $l \in \{1,2,\dots, K-1\}$.

Even if the dynamics induced by $\mathcal{G}$ has some symmetry (in fact, it is symmetric when $\theta = 1$), we prove that $\big(\eta^{(N)}_t\big)_{t \ge 0}$ is not reversible when $K \ge 4$ or when $K = 3$ and $\theta \neq 1$. However, we show that the stationary distribution of the $N$-particle process is rotation invariant. Using this invariance, we calculate the mean of the proportion of particles in each state under the stationary distribution.
\begin{thm}[Non-reversibility and rotation invariance]\label{big_theo}
	The $N$-particle system with generator given by (\ref{generator1}) has the following properties
	\begin{enumerate}
		\item[a)] It is not reversible, except when $K=3$ and $\theta = 1$.
		\item[b)] Its stationary distribution, denoted by $\nu_N$, is
		invariant by rotations, i.e.\
		\[
		\nu_{N} = \nu_N \circ \phi^{(l)},\quad l \in \{1,2,\dots, K-1\}.
		\]
		\item[c)] Under the stationary dynamics, the empirical distribution
		of the $N$-particle system is an unbiased estimator of the $\mathrm{QSD}$ of
		$(Z_t)_{t \ge 0}$, i.e.\
		\[
		\mathbb{E}_{\nu_N} \left[ \frac{\eta(k)}{N}  \right]  = \frac{1}{K},
		\quad k \in \mathbb{Z}/K \mathbb{Z}.
		\]
	\end{enumerate}
\end{thm}

Theorem  \ref{big_theo} is proved in Section \ref{sec:covariances_stationary}. Using parts $b)$ and $c)$ of Theorem \ref{big_theo}, the following result is immediate.

\begin{cor}[Cyclic symmetry]\label{lemma3}
	For every $K \ge 3$ we have
	\begin{equation*}\label{ec6}
		\operatorname{Cov}_{\nu_{N}}\left[ \frac{\eta(0)}{N}, \frac{\eta(k)}{N} \right]
		= \operatorname{Cov}_{\nu_{N}}\left[\frac{\eta(0)}{N}, \frac{\eta(K-k)}{N}
		\right],
		\quad k \in \mathbb{Z}/K\mathbb{Z}.
	\end{equation*}
\end{cor}

Let $T_n$ and $U_n$ be the $n$-th degree Chebyshev polynomials of first and second kind, respectively, for $n \ge 1$. We recall that polynomials $\big(T_n\big)_{n \ge 0}$ and $\big(U_n\big)_{n \ge 0}$ satisfy both the recurrence relation
\begin{equation}\label{recur_cheby}
	p_{n+1}(x) = 2x \, p_n(x) - p_{n-1}(x), \text{ for all } n \ge 1,
\end{equation}
with initial conditions $T_0(x) = U_0(x) = 1$, $T_1(x) = x$ and $U_1(x) = 2x$, see e.g.\ \cite{mason2002chebyshev}. 
We also extend the definition of the Chebyshev polynomials of second kind for $n=-1$, by putting $U_{-1}(x) = 0$.

The following theorem provides explicit expressions for $\operatorname{Cov}_{\nu_N} \left[ {\eta(0)}/{N}, {\eta(k)}/{N} \right]$ in terms of the Chebyshev polynomials of first and second kind, for $k \in \{ 0,1,\dots, K-1 \}$ and the constant $\beta_N$, defined by
\begin{equation} \label{eq:params}
	\beta_N  = 2\left( 1+  \frac{p}{(N-1)(1+\theta)} \right).
\end{equation}

\begin{thm}[Explicit expressions for the covariances] \label{thm:sk_expl}
	We have
	\begin{itemize}
		\item If $K= 2 K_2$, $K_2 \geq 2$,
		\begin{eqnarray}
			\operatorname{Var}_{\nu_N}\left[ \frac{\eta(0)}{N} \right] &=& \frac{N-1}{KN} \frac{2  }{\beta_N+2} \frac{T_{K_2}(\beta_N/2) }{ U_{K_2 - 1}(\beta_N/2)  } + \frac{1}{KN}  - \frac{1}{K^2},\label{sys1:s0even}\\
			\operatorname{Cov}_{\nu_N} \left[ \frac{\eta(0)}{N}, \frac{\eta(k)}{N} \right] &=& \frac{N -1}{KN} \frac{2  }{\beta_N+2} \frac{T_{K_2-k}(\beta_N/2) }{U_{K_2- 1}(\beta_N/2)  } - \frac{1}{K^2}, \label{sys1:skeven}
		\end{eqnarray}
		for all $1 \leq k \leq K_2-1$.
		\item If $K = 2 K_2 + 1$, $K_2 \geq 1$,
		\begin{eqnarray}
			\operatorname{Var}_{\nu_N}\left[ \frac{\eta(0)}{N} \right] &=&   \frac{N-1}{KN} \frac{U_{K_2}(\beta_N/2) - U_{K_2-1}(\beta_N/2)  }{U_{K_2}(\beta_N/2) + U_{K_2-1}(\beta_N/2)  } + \frac{1}{KN}  - \frac{1}{K^2},\label{sys1:s0odd}\\
			\operatorname{Cov}_{\nu_N} \left[ \frac{\eta(0)}{N}, \frac{\eta(k)}{N} \right] &=&   \frac{N - 1}{KN} \frac{U_{K_2-k}(\beta_N/2) - U_{K_2-k-1}(\beta_N/2)  }{ U_{K_2}(\beta_N/2) + U_{K_2-1}(\beta_N/2)  }  - \frac{1}{K^2},\label{sys1:skodd}
		\end{eqnarray}
		for all $1 \leq k \leq K_2$.
	\end{itemize}
	
\end{thm}

Theorem \ref{thm:sk_expl} is proved in Section \ref{sec:thm:sk_expl}.
Using previous result it is possible to show that the covariance between the proportions of particles under the stationary distribution in two different states decreases as a function of the graph distance between the states.

\begin{cor}[Geometry of the cycle graph and covariances]\label{corollary_geometry}
	The covariance between two states under the stationary measure, $\nu_N$, is decreasing as a function of the graph distance between these states, i.e.\ for all $k=0,1,\dots, \lfloor \frac{K}{2} \rfloor - 1$ we have
	$$\operatorname{Cov}_{\nu_N}\left[ \frac{\eta(0)}{N}, \frac{\eta(k)}{N} \right] \ge \operatorname{Cov}_{\nu_N}\left[ \frac{\eta(0)}{N}, \frac{\eta(k+1)}{N} \right].$$
\end{cor}

With the aim of proving the convergence of the proportion of particles in each state to $1/K$, we study the behavior of $\operatorname{Var}_{\nu_N}[\eta(0)/N]$ as a function of $1/N$ when $N$ tends to infinity. Theorem 2 in \cite{asselah2011finite} states that these variances vanishes when $N$ goes to infinity. 
We thus focus on the speed of this convergence.
For this purpose, we find the asymptotic development of second order for $\operatorname{Cov}_{\nu_N}\left[ {\eta(0)}/{N}, {\eta(k)}/{N} \right]$ as a function of $1/N$ when $N$ tends to infinity, for $k \in \mathbb{Z}/ K \mathbb{Z}$.

\begin{thm}[Asymptotic development of two-particle covariances]\label{thm_two_particles_covariance}
	The asymptotic series expansion of order $2$ when $N \to +\infty$ of $\operatorname{Cov}_{\nu_N}\left[ \frac{\eta(0)}{N}, \frac{\eta(k)}{N} \right]$, for $k \in \mathbb{Z}/ K \mathbb{Z}$, is given by
	{\small
		\begin{align}
			\operatorname{Cov}_{\nu_N}\left[ \frac{\eta(0)}{N}, \frac{\eta(k)}{N} \right] &=  \frac{1}{K N}
			\left( \mathbbm{1}_{\{k=0\}} -\frac{1}{K} + \frac{6k(k - K) + K^2-1}{6 K} \frac{p}{1 + \theta} \right) \nonumber \\
			&+ \frac{1}{K^2 N^2}\frac{30k(K-k)[k(K-k)+2] - (K^2-1)(K^2+11)}
			{180 } \left( \frac{p}{1 + \theta} \right)^2 + o \left( \frac{1}{N^2} \right).  \label{DLcovariancesK}
		\end{align}
	}
\end{thm}

The following result provides a bound for the speed of convergence of the empirical distribution induced by the $N$-particle system to the QSD when $N \rightarrow \infty$.
\begin{cor}[Convergence to the QSD]\label{corol:Convergence_QSD}
	We have
	\begin{equation}\label{ineq:bounding_norm2_corol}
		\mathbb{E}_{\nu_{N}}\big[ \left\| m(\eta) - \nu_{\mathrm{qs}} \right\|_{2} \big] \leq \sqrt{\frac{K-1}{N}} \sqrt{1 + \frac{p(K+1)}{6 (1 + \theta)}} + o\left( \frac{1}{\sqrt{N}}\right).
	\end{equation}
\end{cor}

Theorem \ref{thm_two_particles_covariance} and Corollary \ref{corol:Convergence_QSD} are proved in Section \ref{subsec_thm_assymptotic}.
In particular, Corollary \ref{corol:Convergence_QSD} implies the convergence at rate $1/\sqrt{N}$ under the stationary distribution of $m(\eta)$ towards the uniform distributions, when $N \to \infty$. 
Cloez and Thai \cite[Cor.\ 2.10]{cloez2016fleming} provide the same rate of convergence for the Fleming\,--\,Viot process in the $K$-complete graph.
Moreover, Champagnat and Villemonais \cite[Thm.\ 2.3]{ChampagnatVillemonaisP2018c} provide a general rate of convergence $1/N^{\alpha}$, with $\alpha = \frac{\gamma}{2(\|\kappa\|_{\infty} + \gamma)}$. 
In particular, as soon as $\|\kappa\|_{\infty} \neq 0$, one has $\alpha < 1/2$, which is actually not the optimal rate for the asymmetric random walk, killed at a uniform rate, studied in this paper.
To the best of our knowledge, there are no general results on Fleming\,--\,Viot process in discrete spaces assuring the rate of convergence $1/\sqrt{N}$, under the stationary distribution, of the empirical distribution to the QSD.

Finally, in Section \ref{sec:ergo} we study the convergence of the empirical distribution, $m(\eta_t)$, to the quasi-stationary distribution of $(Z_t)_{t \ge 0}$ when $t$ tends to infinity.
Let us denote by $\overline{m}\big(\eta_t^{(N)}\big)$ the empirical mean measure induced by the $N$-particle process at time $t$, defined by $\overline{m}\big(\eta_t^{(N)}\big)(k) = \mathbb{E}\big[m\big(\eta_t^{(N)}\big)(k)\big] = \mathbb{E}\big[\eta_t^{(N)}(k)/N\big]$. Using (\ref{exprLBk}) we can prove the following two theorems.

\begin{thm}[Mean empirical distribution]\label{ergodicity_thm}
	Consider $\eta \in \mathcal{E}_{K,N}$ and $\big(\eta_t^{(N)}\big)_{t \ge 0}$ the $N$-particle process with initial distribution concentrated at $\eta$. We have
	\[
	\overline{m}\big(\eta_t^{(N)}\big) = \mathcal{L}_{m(\eta)}(Z_t \mid t < \tau_p).
	\]
	Furthermore, for every  probability measure $\nu$ on $\mathbb{Z}/K \mathbb{Z}$ we obtain
	\begin{equation}
		\left|  \varphi_{m(\eta)}\left( \frac{2 \pi}{K} \right) - \varphi_{\nu} \left( \frac{2 \pi}{K} \right) \right| \mathrm{e}^{ - \rho_K t } 
		\le \left\| \overline{m}\big(\eta_t^{(N)}\big) - \mathcal{L}_{\nu}(Z_t \mid t < \tau_p) \right\|_2 
		\le \left\|  m(\eta) - \nu \right\|_2 \mathrm{e}^{ - \rho_K t }, \label{bound_ergo_L2}
	\end{equation}
	where $\rho_K$ are defined by (\ref{def:rhoK}), and $\varphi_{m(\eta)}$ and $\varphi_{\nu}$ denote the characteristic functions associated to the distributions $m(\eta)$ and $\nu$, respectively.
\end{thm}

Thus, the proportion of particles in each state is an unbiased estimator of the distribution of the conditioned process for all $t \ge 0$. Using \cite[Thm.\ 1.2]{ferrari2007countable} we know that the variance of the proportion of particles in each state at time $t \ge 0$ vanishes when $N$ goes to infinity, for every $t \ge 0$. The following result provides a bound for this convergence.

\begin{thm}[Convergence to the Conditioned Process]\label{ergodicity_thm_2}
	We have the following uniform upper bound for the variance of the proportion of particles in each state
	\begin{equation}
		\hspace*{- 8 pt} \max\limits_{\substack{\eta \in \mathcal{E}_{K,N} \\ k \in \mathbb{Z}/K \mathbb{Z}}} \left| \operatorname{Var}_{\eta} \left[ \frac{\eta_t^{(N)}(k)}{N} \right] - \operatorname{Var}_{\nu_N} \left[ \frac{\eta(k)}{N} \right] \right| \le   C_{K,N} \frac{\mathrm{e}^{-p_N t} - \mathrm{e}^{- \rho_K t}}{\rho_K - p_N}  + \mathrm{e}^{- p_N t} \operatorname{Var}_{\nu_N} \left[ \frac{\eta(0)}{N} \right],  \label{thm_var_1}
	\end{equation}
	where $\rho_K$ is given by (\ref{def:rhoK}) and
	\begin{eqnarray}
		p_N &=& \frac{2 p}{N - 1}, \label{def:pN}\\
		C_{K,N} &=& \frac{2}{N} \left( 1 + \theta + \frac{p}{N-1} + \frac{p N(K+1)\sqrt{K-1}}{K\sqrt{K}(N-1)} \right). \label{def:CKN}
	\end{eqnarray}
	Furthermore,
	\begin{align}
		 \left|\varphi_{m(\eta)}(t) - \varphi_{\nu}(t)\right| &\mathrm{e}^{- \rho_K t}  \le \mathbb{E}_{\eta} \left[\left\|m\left(\eta_{t}^{(N)}\right) - \mathcal{L}_{\nu} (Z_t \mid t \le \tau_p) \right\|_2\right] \nonumber\\
		&\leq \sqrt{\frac{K}{N}} \left( D_K \frac{1 - \mathrm{e}^{- \rho_K t}}{\rho_K}  + E_K  \right)^{1/2}  + \mathrm{e}^{- \rho_K t} \|m(\eta) - \nu\|_2 + o\left( \frac{1}{\sqrt{N}} \right), \label{thm_var_2}
	\end{align}
	for every $\eta \in \mathcal{E}_{K,N}$ and every initial distribution $\nu$ on $\mathbb{Z}/K \mathbb{Z}$,
	where $\rho_K$ is given by (\ref{def:rhoK}), and
	\begin{equation}
		D_K = 2 \left(1 + \theta +  \frac{p(K+1)\sqrt{K-1}}{K\sqrt{K}} \right), \;\;\;\;	E_K = \frac{K-1}{K^2} + \frac{K^2 - 1}{6 K^2 (1 + \theta)}. \label{def:DEK}
	\end{equation}
\end{thm}
Theorems \ref{ergodicity_thm} and \ref{ergodicity_thm_2} is proved in Section \ref{sec:ergo}. Similar results are proved in \cite{cloez2016fleming} for the Fleming\,--\,Viot process on the complete graph and for the two-point process.

\begin{rmk}[Uniform bound]
	Note that the bound given by (\ref{thm_var_1}) tends exponentially towards zero when $t \rightarrow \infty$. In particular, the right side of (\ref{thm_var_1}) is bounded in $t$ and can be used to obtain a uniform bound for the variance of the proportion of particles in each state of order $1/N$. Namely, using (\ref{thm_var_1}) and the inequality
	\(
	(\mathrm{e}^{-p_N t} - \mathrm{e}^{- \rho_K t})/(\rho_K - p_N) \le 1/\max(\rho_K, p_N),
	\)
	we obtain
	\[
	\sup_{t \ge 0} \max\limits_{\substack{\eta \in \mathcal{E}_{K,N} \\ k \in \mathbb{Z}/K \mathbb{Z}}}\operatorname{Var}_{\eta} \left[ \frac{\eta_t^{(N)}(k)}{N} \right] \le    \frac{C_{K,N}}{\max(\rho_K, p_N)} +  2 \operatorname{Var}_{\nu_N} \left[ \frac{\eta(0)}{N} \right]
	= \left( \frac{D_K}{\rho_K} + 2 E_K \right) \frac{1}{N} + o\left( \frac{1}{N} \right),
	\]
	where $\rho_K$, $p_N$, $C_{K,N}$ and $D_k$ and $E_k$, are given by (\ref{def:rhoK}), (\ref{def:pN}), (\ref{def:CKN}) and (\ref{def:DEK}), respectively.
	
	Similar bounds are obtained for the convergence to the conditional distribution for Fleming\,--\,Viot process in discrete state spaces, see e.g.\ \cite[Thm.\ 1.1]{Moral2000} and \cite[Thm.\ 2.2]{Villemonais_2014}. However, these results are not uniform in $t \ge 0$. Corollary 1.5 in \cite{cloez2016quantitative} does provide a uniform bound under certain conditions of order $1/N^{\gamma}$, with $\gamma < 1/2$, for the $1$-distance between the empirical law associated to the Fleming\,--\,Viot process at time $t$ and the law of the conditioned process. 
	However, this result does not hold for the Fleming\,--\,Viot process on the $K$-cycle graph we study here, for $K \ge 6$, since the parameter $\lambda$ given by (\ref{def:lambda_Cloez}) is null.
\end{rmk}

The rest of this paper is organized as follows. Section \ref{sec:N_particles} gives the proof of Theorem \ref{thm1}. In Section \ref{sec:covariances_stationary} we study the covariances of the proportions of particles in each state under the stationary distribution, and 
we thus prove Theorems \ref{big_theo}, \ref{thm:sk_expl} and \ref{thm_two_particles_covariance}. Finally, Section \ref{sec:ergo} is devoted to the proof of Theorems \ref{ergodicity_thm} and \ref{ergodicity_thm_2} related to the variance of the proportion of particles in each site at a given time $t \ge 0$.

%%%%======================================================

\section{The asymmetric random walk on the cycle graph} \label{sec:N_particles}

We first prove that the QSD of $(Z_t)_{t \ge 0}$, denoted by $\nu_{\mathrm{qs}}$, which is the stationary distribution of $(X_t)_{t \ge 0}$, is the uniform distribution on $\mathbb{Z}/K \mathbb{Z}$. We also provide exponential bounds for the speed of convergence in the $2$-distance and the total variation distance of $\mathcal{L}_{\nu}(Z_t \mid t < \tau_p)$ to $\nu_{\mathrm{qs}}$.

Recall that a square matrix $C$ is called circulant if it takes the form
\begin{equation}\label{def:circulant}
	C =
	\left(
	\begin{array}{ccccc}
		{c_{0}} & {c_{1}} & {\ldots} & {c_{n-2}} & {c_{n-1}} \\
		{c_{n-1}} & {c_{0}} & \ddots & c_{n-3} & {c_{n-2}} \\
		{\vdots} & \vdots & \ddots & \ddots & \vdots \\
		{c_{2}} & c_{3} & {\ddots} & c_0 & {c_{1}} \\
		{c_{1}} & {c_{2}} & {\ldots} & {c_{n-1}} & {c_{0}}
	\end{array}
	\right).
\end{equation}
It is evident that a circulant matrix is completely determined by its first row, therefore we will denote a circulant matrix with the form given by (\ref{def:circulant}) by $C = \operatorname{circ}(c_0, c_1, \dots, c_{n-1})$.

Let $Q$ be the infinitesimal generator matrix of the process $(X_t)_{t \ge 0}$. Then, $Q$ is circulant and it satisfies
\begin{equation}\label{Qmatrix_def}
	Q = \operatorname{circ}(-(1 + \theta), 1, 0, \dots, 0, \theta).
\end{equation}
Let us also denote by $\mathrm{i}$ the complex root of $-1$.
Since the matrix \(Q\)
is circulant, its spectrum is explicitly known, as follows in the next lemma.
\begin{lem}[Spectrum of $Q$]\label{spectrum_Q}
	The matrix $Q$ satisfies $Q = F_K \Lambda F_{K}^{\star}$, where
	\begin{itemize}
		\item $F_K$ is the $K$-dimensional Fourier matrix, i.e.\ the unitary matrix defined by
		\begin{equation}\label{eq:FourierMatrix}
			[F_K]_{r,c} = \frac{1}{\sqrt{K}}(\omega_K)^{-r \, c},
		\end{equation}
		for each $r,c \in \{0,1,\dots,K-1\}$, where $\omega_K = \mathrm{e}^{\mathrm{i} \frac{2 \pi}{K}}$,
		\item $F_K^{\star}$ is the conjugate of $F_K$ (and also its inverse because $F_K$ is unitary and symmetric),
		\item $\Lambda$ is the $K \times K$ diagonal matrix with $[\Lambda]_{k,k} = \lambda_k$, for all $0 \leq k \leq K-1$, where \[\lambda_k = -(1 + \theta)\sin^2\left( \frac{\pi k}{K} \right) + \mathrm{i}(1-\theta) \sin\left( \frac{ 2 \pi k}{K} \right),\]
		for $k =0,1, \dots, K-1$.
	\end{itemize}
\end{lem}

\begin{proof}[Proof of Lemma \ref{spectrum_Q}]
	Let us define the polynomial $p_Q: s \mapsto  -(1 + \theta) + s + \theta s^{K-1}$. Since $Q$ is a circulant matrix, we can use \cite[Thm.\ 3.2.2]{davis_circulant} to diagonalize $Q$ in the following way
	\[ Q = F_K \operatorname{Diag}(\lambda_0, \lambda_1, \dots, \lambda_{K-1}) F_K^{\star},\]
	where $F_K$ is the Fourier matrix defined by (\ref{eq:FourierMatrix}) and
	\begin{align*}
		\lambda_k &= p_Q\big(\mathrm{e}^{\mathrm{i} \frac{2 k \pi}{K}}\big) = -(1 + \theta) + \mathrm{e}^{\mathrm{i} \frac{2 k \pi}{K}} + \theta \left(\mathrm{e}^{\mathrm{i} \frac{2 k \pi}{K}}\right)^{K-1}\\
		&= -(1 + \theta) \left[ 1 - \cos\left( \frac{2 \pi k}{K} \right) \right] + \mathrm{i} (1- \theta) \sin \left( \frac{2 \pi k}{K} \right)\\
		&= -2 (1 + \theta) \sin^2\left( \frac{\pi k}{K} \right) + \mathrm{i} (1- \theta) \sin \left( \frac{2 \pi k}{K} \right),\\
	\end{align*}
	for $k = 0,1,\dots, K-1$.
\end{proof}

\begin{rmk}[Eigenvalues of $Q$]
	Note that
	\(
	\frac{[\Re (\lambda_k) + (1 + \theta)]^2}{(1 + \theta)^2} + \frac{[\Im (\lambda_k)]^2}{(1 - \theta)^2} = 1,
	\)
	for all $\theta \neq 1$, where $\Re(\lambda_k)$ and $\Im(\lambda_k)$ are the real and the imaginary parts of $\lambda_k$, respectively, for $k = 0,1,\dots, K-1$. Thus, all the eigenvalues $\lambda_k$ are on the ellipse with center $(0,-(1+\theta))$ and  equation $$\frac{(x + 1 + \theta)^2}{(1 + \theta)^2} + \frac{y^2}{(1 - \theta)^2} = 1.$$
	Of course, for $\theta = 1$, since the matrix $Q$ is symmetric, all the eigenvalues are real.
	
	Also, the \textit{second largest eigenvalue in modulus} ($\mathrm{SLEM}$) of $Q$, denoted by $\rho_K$, is given by (\ref{def:rhoK}) and it is reached for $-\Re(\lambda_1)$ and $-\Re(\lambda_{K-1})$. 
	The minimum of $\Re(\lambda_k)$ is reached for $\Re(\lambda_{K/2})$, if $K$ is even, and for $\Re(\lambda_{(K-1)/2})$ and $\Re(\lambda_{(K+1)/2})$, if $K$ is odd.
\end{rmk}

\subsection{Proof of Theorem \ref{thm1}}

\begin{proof}[Proof of Theorem \ref{thm1}]
	
	We know that $ Q  = F_K \Lambda F_K^{\star}$. Therefore
	$ \mathrm{e}^{tQ} = F_K \mathrm{e}^{t \Lambda} F_K^{\star}$,
	and it follows that
	\begin{equation*}
		\mathrm{e}^{tQ} = \sum_{k=0}^{K-1} \mathrm{e}^{\lambda_k t} F_K U_k F_K^{\star} = \sum_{k=0}^{K-1} \mathrm{e}^{\lambda_k t} \Omega_k,
	\end{equation*}
	where $U_k, \; 0 \leq k \leq K-1$, is the $K \times K$ matrix with $[U_k]_{k,k} = 1$ and $0$ elsewhere,
	and $\Omega_k$ is defined as $\Omega_k = F_K U_k F_K^{\star}$. In fact, $\Omega_k$ is the symmetric circulant matrix satisfying
	$[\Omega_k]_{r,c} = \frac{1}{K} \omega^{k(r-c)}$,
	for all $0 \leq r,c \leq K-1$ and for every $k \in \{0,1,\dots, K-1\}$.
	In particular $[\Omega_0]_{r,c} = \frac{1}{K}$ for all $0 \leq r,c \leq K-1$, and $\Omega_k \, \Omega_l = \mathbf{0}$, for all $k \neq l$. 
	Then, for two probability measures $\mu$ and $\nu$ on $\{0,1,\dots,K-1\}$ we have
	\begin{equation}\label{eq:cond_Omega_0}
		(\mu-\nu) \Omega_0 = \mathbf{0}
	\end{equation}
	and therefore
	\begin{equation} \label{eq:orthogonal_descomp}
		(\mu - \nu) \mathrm{e}^{tQ} = \sum_{k=1}^{K-1} \mathrm{e}^{\lambda_k t} (\mu-\nu) \Omega_k.
	\end{equation}
	
	Let us denote by $\langle \cdot, \cdot \rangle$ the usual inner product in $\mathbb{C}$ and for a matrix $A$ let us denote by $A^T$ its transpose. Note that for every $K$-dimensional vector $\mathbf{x}$ and $k \neq l$ we have 
	\[
		\langle \mathbf{x} \, \Omega_k, \; \mathbf{x} \, \Omega_l \rangle = \mathbf{x} \, \Omega_k \big[\big(\Omega_l^T\big)^{\star}\big] \, (\mathbf{x}^{\star})^T = \mathbf{x} \, \Omega_k \Omega_l \, (\mathbf{x}^{\star})^T = 0. 
	\]	
	Thus, the set of vectors $\left( \mathbf{x} \Omega_k \right)_{k = 1}^{K-1}$ are orthogonal in $(\mathbb{C}, \langle \cdot, \cdot \rangle)$. Now, using (\ref{eq:orthogonal_descomp}) and Pythagoras' theorem we have
	\begin{eqnarray*}
		\left\| (\mu - \nu) \mathrm{e}^{tQ} \right\|_2^2 &=& \sum_{k=1}^{K-1} \left\|\mathrm{e}^{\lambda_k t} (\mu-\nu) \Omega_k\right\|^2_2 \nonumber\\
		&=& \sum_{k=1}^{K-1} \mathrm{e}^{2 \Re(\lambda_k) t} \left\| (\mu-\nu) \Omega_k\right\|^2_2. \label{eq:sum_riemman}
	\end{eqnarray*}
	Since $\rho_K = - \max\limits_{k=1,\dots,K-1} \Re(\lambda_k)$ we obtain
	\begin{eqnarray*}
		\left\| (\mu - \nu) \mathrm{e}^{tQ} \right\|_2^2 &\le& \mathrm{e}^{-2 \rho_K t} \sum\limits_{k=1}^{K-1} \left\| (\mu-\nu) \Omega_k\right\|^2_2 \nonumber \\
		& = & \mathrm{e}^{-2 \rho_K t} \left\| \sum\limits_{k=1}^{K-1}  (\mu-\nu) \Omega_k\right\|^2_2 \nonumber \\
		& = & \mathrm{e}^{-2 \rho_K t} \left\| \sum\limits_{k=0}^{K-1}  (\mu-\nu) \Omega_k\right\|^2_2 \nonumber \\
		& = & \mathrm{e}^{-2 \rho_K t} \|\mu - \nu\|_2^2. \label{ineq:norm2Upperbound}
	\end{eqnarray*}
	Note that the first equality holds due the Pythagoras' theorem, the second one uses (\ref{eq:cond_Omega_0}) and the last one uses the fact that
	\[
	\sum\limits_{k=0}^{K-1}  (\mu-\nu) \Omega_k = \mu - \nu.
	\]
	 Note that the upper bound in (\ref{exp_bound_TVdistrib}) is proved using the \emph{Cauchy\,–-\,Schwarz inequality}, which implies
	\[
	\left\|  \Delta_t(\mu,\nu) \right\|_{\mathrm{TV}} \le \frac{\sqrt{K}}{2}  \left\| \Delta_t(\mu,\nu) \right\|_{2},
	\]
	where $\Delta_t(\mu,\nu)$ is as defined in the statement of Theorem \ref{thm1}, and the inequality holds for every pair of distributions $\nu$ and $\mu$ on $\mathbb{Z}/K\mathbb{Z}$, and for all $t \ge 0$.
	
	To prove the lower bounds in (\ref{exp_bound_norm_2distrib}) and (\ref{exp_bound_TVdistrib}) we recall the the $r$-norm of a function $f$ on $\mathbb{Z}/K \mathbb{Z}$, allows the following characterization:
	\[
	\| f \|_r = \max_{g} \frac{|\langle f, g \rangle|}{\|g\|_q},
	\]
	where $q \in [1, \infty]$ is the conjugate of $r \in [1, \infty]$, i.e.\ $1/r + 1/q = 1$, and the maximum runs over all the functions on $\mathbb{Z}/K \mathbb{Z}$.
	Now, take $g : k \in \mathbb{Z}/K \mathbb{Z} \mapsto \frac{1}{\sqrt{K}} (\omega_K)^k$ as a test function, where $\omega_K = \mathrm{e}^{\frac{2 \pi}{K} \mathrm{i}}$.
	Note that viewed as a column vector, $g$ is equal to the last column of the Fourier matrix $F_K$.
	Then, $g$ is a right eigenfuntion of $Q$ with associated eigenvalue $-\rho_K$.
	Moreover,  $\|g\|_2 = 1$ and $\|g\|_{\infty} = 1/\sqrt{K}$. 
	Therefore,
	\begin{align*}
	\| \nu \mathrm{e}^{t Q} - \mu \mathrm{e}^{t Q} \|_2 & \ge \frac{|\langle \nu \mathrm{e}^{t Q} - \mu \mathrm{e}^{t Q}, g \rangle | }{\|g\|_2} = \left| \varphi_{\nu}\left( \frac{2 \pi}{K} \right) - \varphi_{\mu}\left( \frac{2 \pi}{K} \right) \right| \mathrm{e}^{- \rho_K t},\\
	\left\| \nu \mathrm{e}^{t Q}- \mu \mathrm{e}^{t Q} \right\|_{\mathrm{TV}} & \ge \frac{|\langle \nu \mathrm{e}^{t Q} - \mu \mathrm{e}^{t Q}, g \rangle | }{2 \|g\|_\infty} =    \frac{\sqrt{K}}{2} \left| \varphi_{\nu}\left( \frac{2 \pi}{K} \right) - \varphi_{\mu}\left( \frac{2 \pi}{K} \right) \right| \mathrm{e}^{- \rho_K t},\\
	\end{align*}
	To prove (\ref{exp_bound_2norm}) first note that the $2$-distance and the total variation distances satisfy
	\begin{eqnarray*}
		\mathrm{D}_2(t) &=& \max_{ k \in \mathbb{Z}/K \mathbb{Z} } \left\| \mathcal{L}_{k}(Z_t \mid  t < \tau_p) - \nu_{\mathrm{qs}} \right\|_{2},\\
		\mathrm{D}_{\mathrm{TV}}(t) &=& \max_{ k \in \mathbb{Z}/K \mathbb{Z} } \left\| \mathcal{L}_{k}(Z_t \mid  t < \tau_p) - \nu_{\mathrm{qs}} \right\|_{\mathrm{TV}},
	\end{eqnarray*}
	which is a consequence of the convexity of these distances.
	Thus, the upper bounds in expression (\ref{exp_bound_2norm}) and (\ref{exp_bound_TVnorm}) are consequence of the equality $\|\delta_k - \nu_{\mathrm{qs}} \|_2 = \sqrt{\frac{K - 1}{K}}$. 
	The lower bounds in (\ref{exp_bound_2norm}) and (\ref{exp_bound_TVnorm}) is obtained using that $\varphi_{\nu_{\mathrm{qs}}}(2 \pi/K) = 0$ and $\varphi_{\delta_k}(2 \pi/K) = |g(k)| = 1/\sqrt{K}$, for every $k \in \mathbb{Z}/K \mathbb{K}$.

\end{proof}

\section{Covariances of the proportions of particles under the stationary distribution}\label{sec:covariances_stationary}

The following lemma gives us informations about the invariance of the generator $\mathcal{L}_{K,N}$, defined in (\ref{generator1}), by the rotation function $\phi$ defined in (\ref{def_rotation}).

\begin{lem}[Rotation invariance of the generator]\label{lema_rotation_generator}
	The generator $\mathcal{L}_{K,N}$ of $(\eta_t^{(N)})_{t \ge 0}$ satisfies
	\begin{equation} \label{invariance_rotations}
		\mathcal{L}_{K,N} \mathbbm{1}_{\eta} = \mathcal{L}_{K,N} \mathbbm{1}_{\phi(\eta)} \circ \phi,
	\end{equation}
	for every $\eta \in \mathcal{E}_{K,N}$.
\end{lem}

\begin{proof}
	Note that
	\begin{equation}\label{generator_lema_1}
		(\mathcal{L}_{K,N} \mathbbm{1}_{\eta}) (\eta') = \eta'(i) \left( \mathbbm{1}_{\{ j = i + 1\}} + \theta \mathbbm{1}_{\{ j = i - 1\}} + p \frac{\eta'(j)}{N-1} \right),
	\end{equation}
	if $\eta = T_{i \rightarrow j} \eta'$, for some $i,j \in \mathbb{Z}/K \mathbb{Z}$, and it is null otherwise.
	Now, if $\eta = T_{i \rightarrow j} \eta'$, then we have $\phi(\eta) = T_{(i+1) \rightarrow (j+1)} \phi(\eta')$. Thus,
	\begin{equation}\label{generator_lema_2}
		\left(\mathcal{L}_{K,N} \mathbbm{1}_{\phi(\eta)} \right) (\phi(\eta')) = \phi(\eta')(i+1) \left( \mathbbm{1}_{\{ j = i + 1\}} + \theta \mathbbm{1}_{\{ j = i - 1\}} +p \frac{ \phi(\eta')(j+1)}{N-1} \right).
	\end{equation}
	Using (\ref{generator_lema_1}) and (\ref{generator_lema_2}) we can see that (\ref{invariance_rotations}) holds, since $\eta'(i) = \phi(\eta')(i+1)$ and $\eta(j) = \phi(\eta)(j+1)$.
\end{proof}

\subsection{Proof of Theorem \ref{big_theo}}

We will now prove Theorem \ref{big_theo}, which describes some properties of $\nu_N$, the stationary distribution of the $N$-particle process $\big(\eta_t^{(N)}\big)_{t \ge 0}$.

\begin{proof}[Proof of Theorem \ref{big_theo}]
	\begin{itemize}
		\item[]
		\item[a)] \emph{The process $\big(\eta_t^{(N)}\big)_{t \ge 0}$ is not reversible, except when $K=3$ and $\theta = 1$.}
		
		For $K=3$ and $N \geq 2$, let us consider the three states in $\mathcal{ E}_{3,N}$,
		\[
		\eta_1 = [N  ,0,0],\;\;\;
		\eta_2 = [N-1,1,0],\;\;\;
		\eta_3 = [N-1,0,1].
		\]
		It is straightforward to verify that
		\[	
		\begin{array}{lll}
			(\mathcal{L}_{K,N} \mathbbm{1}_{\eta_2}) (\eta_1) = N, & (\mathcal{L}_{K,N} \mathbbm{1}_{\eta_3}) (\eta_1) = N\theta, & (\mathcal{L}_{K,N} \mathbbm{1}_{\eta_1}) (\eta_2) =p+\theta, \\
			(\mathcal{L}_{K,N} \mathbbm{1}_{\eta_3}) (\eta_2) = 1, & (\mathcal{L}_{K,N} \mathbbm{1}_{\eta_1}) (\eta_3) =p+1, & (\mathcal{L}_{K,N} \mathbbm{1}_{\eta_2}) (\eta_3) = \theta.
		\end{array}
		\]
		Moreover,
		\begin{align*}
			(\mathcal{L}_{K,N} \mathbbm{1}_{\eta_1}) (\eta_3) \cdot 
			(\mathcal{L}_{K,N} \mathbbm{1}_{\eta_2}) (\eta_1) \cdot 
			(\mathcal{L}_{K,N} \mathbbm{1}_{\eta_3}) (\eta_2) &= (p+1)N, \\
			(\mathcal{L}_{K,N} \mathbbm{1}_{\eta_3}) (\eta_1) \cdot
			(\mathcal{L}_{K,N} \mathbbm{1}_{\eta_2}) (\eta_3) \cdot
			(\mathcal{L}_{K,N} \mathbbm{1}_{\eta_1}) (\eta_2) &= N \theta^2(p + \theta),
		\end{align*}
		the Kolmogorov cycle reversibility criterion, see \cite[Thm.\ 1.8]{kelly-1979},
		is not satisfied unless $\theta=1$. Indeed, note that a necessary
		condition to have reversibility is that the polynomial
		$$\alpha(\theta) = \theta^3 + p (N-1) \theta^2 - p(N-1) - 1 = (\theta -
		1)[\theta^2 + (\theta + 1)(p + 1)]$$ is equal to zero. Now, since
		$\theta^2 + (\theta + 1)(p + 1) > 0$ for all $\theta \ge 0$, the
		polynomial $\alpha(\theta)$ only has one positive root, which is
		$\theta = 1$.
		
		For $K\geq 4, \;N \geq 2$ and $p>0$, let us consider the two states in
		$\mathcal{ E}_{K,N}$: $\eta_1 = [N,0,\dots,0]$ and
		$\eta_2 = [N-1,0,1,\dots,0]$. Because
		$(\mathcal{L}_{K,N} \mathbbm{1}_{\eta_2}) (\eta_1)= 0$ and
		$(\mathcal{L}_{K,N} \mathbbm{1}_{\eta_1}) (\eta_2) =p \neq 0$, the detailed
		balanced property for a reversible process, see \cite[Thm.\
		1.3]{kelly-1979},
		$\nu_N(\eta_1) (\mathcal{L}_{K,N} \mathbbm{1}_{\eta_2}) (\eta_1) = \nu_N(\eta_2)
		(\mathcal{L}_{K,N} \mathbbm{1}_{\eta_1}) (\eta_2)$, is not satisfied.
		
		Therefore, $a)$ is proved except in the special case $K=3$, $N \geq 2$ and
		$\theta = 1$. Note that in this case the model is a complete graph model,
		which was proved to be reversible in \cite[Thm.\ 2.4]{cloez2016fleming}.

		\item[b)] \emph{The stationary distribution $\nu_N$ is invariant by rotation.}

		Since $\nu_N$ is the unique stationary distribution of $\big(\eta_t^{(N)}\big)_{t \ge 0}$,  we know that $\nu_N (\mathcal{L}_{K,N} f) = 0$ for every function $f$ on $\mathcal{E}_{K,N}$.
		Thus, in order to prove that $\nu_N$ is invariant by rotation, it is sufficient to prove that $\nu_N \circ \phi$ also satisfies $(\nu_N \circ \phi) (\mathcal{L}_{K,N} f) = 0$ for every function $f$ on $\mathcal{E}_{K,N}$. Since $\mathcal{E}_{K,N}$ is finite, it is enough to consider the indicator functions $\mathbbm{1}_{\eta}$, for every $\eta \in \mathcal{E}_{K,N}$. Using Lemma \ref{lema_rotation_generator}, we have
		\[
			(\nu_N \circ \phi) (\mathcal{L}_{K,N} \mathbbm{1}_{\eta}) = \nu_N \left( \mathcal{L}_{K,N} \mathbbm{1}_{\eta} \circ \phi^{-1} \right) = \nu_{N} \left( \mathcal{L}_{K,N} \mathbbm{1}_{\phi(\eta)} \right) = 0,
		\]
		for every $\eta \in \mathcal{E}_{K,N}$, where the second equality holds due to (\ref{invariance_rotations}) and the third is due to the fact that $\nu_N$ is stationary for $\mathcal{L}_{K,N}$. Consequently, by the uniqueness of the stationary distribution, we have $\nu_N = \nu_N \circ \phi$. The result trivially holds for any rotation $\phi^{(l)}$, $l \ge 1$.
		
		\item[c)] \emph{Mean of the proportion of particles in each state.}
		
		Using part $b)$ we have
		\( \mathbb{E}_{\nu_{N}}[\eta(0)] = \mathbb{E}_{\nu_{N}}[\phi^{(k)}(\eta)(0)] = \mathbb{E}_{\nu_{N}}[\eta(k)]\),	for all $k =0,1,\dots, K-1$. Also, we know that $\eta(0) + \eta(1) + \dots + \eta(K-1) = N$. Thus, $\mathbb{E}_{\nu_{N}}[\eta(k)] = \displaystyle \frac{N}{K}$, for all $k =0,1,\dots, K-1$.
	\end{itemize}
\end{proof}

Let us define the functions $f_k$ and $f_{k,l}$ on $\mathcal{E}_{K,N}$ as $f_k: \eta \mapsto \eta(k)$ and $f_{k,l}: \eta \mapsto \eta(k) \eta(l)$, for all $k,l \in \{0,1, \dots, K-1\}$.
The following lemma provides explicit expressions for the evaluation of the generator of the $N$-particle process on  these functions.

\begin{lem}[Dynamics of the $N$-particle process]\label{teo_dynam}
	We have that
	\begin{eqnarray}
		\mathcal{L}_{K,N} f_k  &=& f_{k-1} - (1+ \theta) f_k + \theta f_{k+1},\label{exprLBk}\\
		\mathcal{L}_{K,N} f_{k,k}  &=& 2 \left[  f_{k-1,k} - \left(1 + \theta + \frac{p}{N-1} \right) f_{k,k} + \theta f_{k,k+1} \right]  \nonumber \\
		& & + f_{k-1} + \left(1 + \theta + \frac{2 p N }{N-1} \right) f_k + \theta f_{k+1},\label{exprLBkk}\\
		\mathcal{L}_{K,N} f_{k, k+1} &=& - 2 \left(1 + \theta + \frac{p}{N-1} \right) f_{k,k+1} + f_{k-1, k+1} + \theta f_{k+1,k+1} + f_{k,k} + \theta f_{k, k+2}  \nonumber \\
		& &  - f_{k} - \theta f_{k+1},\label{exprLBkl}\\
		\mathcal{L}_{K,N} f_{k, l} &=&  - 2\left(1 + \theta + \frac{p}{N-1}\right) f_{k,l} + f_{k-1, l} + \theta f_{k+1,l} + f_{k,l-1} + \theta f_{k, l+1},\label{exprLBkl1}
	\end{eqnarray}
	for all $k,l \in \mathbb{Z}/K \mathbb{Z}$ such that $|k-l| > 2$.
\end{lem}

The proof of Lemma \ref{teo_dynam} is mostly technical and it is deferred to \ref{app:generator}. The expression (\ref{exprLBk}) given by this lemma is used to study the behavior of the mean of the proportion of particles in each state. Also, (\ref{exprLBkk}), (\ref{exprLBkl}) and (\ref{exprLBkl1}) are used to study the covariances of the number of particles when $t$ and $N$ tend to infinity.

Let us denote
\begin{eqnarray}
	s_k &=& \mathbb{E}_{\nu_N}\left[\frac{f_{l, l+k}(\eta)}{N^2}\right] = \mathbb{E}_{\nu_N}\left[\frac{f_{0,k}(\eta)}{N^2}\right] = \mathbb{E}_{\nu_N}\left[\frac{\eta(0)}{N} \frac{\eta(k)}{N} \right], \label{sl_def}
\end{eqnarray}
for all $k, l \in \mathbb{Z}/K \mathbb{Z}$. Note that the second equality comes from part $b)$ of Theorem \ref{big_theo}. 
Let us define the constant
\begin{equation} \label{eq:params_gamma}
	\gamma_N = -2 \left( 1+  \frac{N p}{(N-1)(1+\theta)} \right).
\end{equation}
The following two lemmas will be useful for obtaining explicit expressions for the quantities $s_k$, for $k=0,1,\dots,K-1$.

\begin{lem}\label{lemma33}
	Then, for $K \geq 3$, the values $s_k$, for $0 \leq k \leq K-2$, satisfy the following linear system:
	\begin{eqnarray}
		- s_{K-1} + \beta_N \, s_0 - s_1 &=& - \frac{\gamma_N}{K N},  \label{ec1} \\
		- s_0  + \beta_N \, s_1  - s_2  &=& - \frac{1}{K N}, \label{ec2}
	\end{eqnarray}
	and when $K \ge 4$:
	\begin{equation}
		- s_{l-1} + \beta_N \, s_l - s_{l+1} = 0,  \label{ec3}
	\end{equation}
	for $2 \leq l \leq K-2$, where $\beta_N$ and $\gamma_N$ are defined by (\ref{eq:params}) and (\ref{eq:params_gamma}), respectively.
\end{lem}

\begin{proof}[Proof of Lemma \ref{lemma33}]
	Using (\ref{exprLBkk}) we have
	\begin{eqnarray*}
		\mathbb{E}_{\nu_N}\big[(\mathcal{L}_{K,N} f_{k,k})(\eta)\big] &=& 2 \left[ \mathbb{E}_{\nu_N}\big[ f_{k-1,k}(\eta)\big] - \left(1 + \theta + \frac{p}{N-1} \right) \mathbb{E}_{\nu_N}\big[ f_{k,k}(\eta)\big] + \theta \mathbb{E}_{\nu_N}\big[ f_{k,k+1}(\eta)\big] \right] +  \nonumber \\
		& & \mathbb{E}_{\nu_N}\big[ f_{k-1}(\eta)\big] + \left(1 + \theta + \frac{2 p N }{N-1} \right) \mathbb{E}_{\nu_N}\big[ f_{k}(\eta)\big] + \theta \mathbb{E}_{\nu_N}\big[ f_{k+1}(\eta)\big].
	\end{eqnarray*}
	Since $\nu_N$ is the stationary distribution, we know that $\mathbb{E}_{\nu_N}\big[(\mathcal{L}_{K,N} f)(\eta)\big] = 0$, for all $f$ on $\mathcal{E}_{K,N}$.
	Thus, using parts $a)$ and $b)$ of Theorem \ref{big_theo} and dividing by $N^2$, we have the equality
	\[
	2 (1+\theta) s_1 - 2 \left( 1+ \theta + \frac{p}{N-1} \right) s_0 = - \frac{2}{KN} \left( 1+ \theta + \frac{pN }{N-1} \right).
	\]
	Dividing by $(1+\theta)$, this last equality is equivalent to
	\begin{equation}\label{ec1-aux}
		\beta_N \, s_0 - 2 s_1 = - \frac{\gamma_N}{KN}.
	\end{equation}
	Note that $s_1 = s_{K-1}$ due to Corollary \ref{lemma3}. 
	Using this fact, we deduce that (\ref{ec1-aux}) is equivalent to (\ref{ec1}).

	Furthermore, using (\ref{exprLBkl}) we get
	\begin{eqnarray*}
		\mathbb{E}_{\nu_N}\big[(\mathcal{L}_{K,N} f_{k,k+1})(\eta)\big] &=&
		- 2 \left(1 + \theta + \frac{p}{N-1} \right) \mathbb{E}_{\nu_N}\big[ f_{k,k+1}(\eta)\big] + \mathbb{E}_{\nu_N}\big[ f_{k-1,k+1}(\eta)\big]  \\
		& & + \theta \; \mathbb{E}_{\nu_N}\big[ f_{k+1,k+1}(\eta)\big] + \mathbb{E}_{\nu_N}\big[ f_{k,k}(\eta)\big] + \theta \; \mathbb{E}_{\nu_N}\big[ f_{k,k+2}(\eta)\big]   \\
		& & - \mathbb{E}_{\nu_N}\big[ f_{k}(\eta)\big] -\theta \; \mathbb{E}_{\nu_N}\big[ f_{k+1}(\eta)\big].
	\end{eqnarray*}
	In a similar way to the previous case we obtain the equation 
	\(
		- s_0 + \beta_N s_1 - s_2 = - {1}/{KN},
	\)
	which is equivalent to (\ref{ec2}).
	
	Similarly, using (\ref{exprLBkl1}), the equality (\ref{ec3}) is proved for all $2 \le l \le K-2$.
\end{proof}
Note that using Corollary \ref{lemma3} and formula (\ref{ec2}) we can obtain the following relation
\begin{equation}\label{eq-K-1}
	-s_{K-2} + \beta_N s_{K-1} - s_0 = - \frac{1}{KN}.
\end{equation}

Let us define the $K \times K$ circulant matrix $A_K$ and the $K$-vector $\mathbf{b}_K$ by
\begin{align*}
	A_K &= \operatorname{circ}(\beta_N, -1, 0, \dots, \dots, 0, -1),\\
	\mathbf{b}_K &= \left(\gamma_N, 1, 0, 0, \ldots, 0, 1 \right)^T,
\end{align*}
for $K \ge 3$, where $\beta_N$ and $\gamma_N$ are defined by (\ref{eq:params}) and (\ref{eq:params_gamma}), respectively.

Using Equations (\ref{ec1}), (\ref{ec2}), (\ref{ec3}) and (\ref{eq-K-1}), the quantities $s_k, \; 0 \leq k \leq K-1$, defined in (\ref{sl_def}) are proved to verify the linear system of equations
\begin{equation} \label{eq:SK}
	A_K \; \mathbf{s}_K = - \frac{1}{K N} \mathbf{b}_K,
\end{equation}
where
\(
\mathbf{s}_K = (s_0, s_1, \dots , s_{K-1})^T \nonumber
\)
and $\beta_N$ and $\gamma_N$ are defined by (\ref{eq:params}).

Note that the vector $\mathbf{b}_K$ is  {\em almost symmetric}, in the sense that    $b_k = b_{K-k}, \; 1 \leq k \leq K-1$, where $b_k, 0 \leq k \leq K-1$, are the $K$ components of $\mathbf{b}_K$. 
Moreover, a vector $\mathbf{b}$ is almost symmetric if and only if the equality $J \mathbf{b} = \mathbf{b}$ holds, where
\[
J = \left(
\begin{array}{cccccc}
	1 & 0 & 0 &  \dots & 0 & 0 \\
	0 & 0 & 0 & \dots  & 0 & 1 \\
	0 & 0 & 0 &   \iddots&  1 & 0 \\
	\vdots & \vdots & \iddots & \iddots & \vdots & \vdots \\
	0 & 0 & 1 & \dots & 0 & 0 \\
	0 & 1 & 0 & \dots & 0 & 0 \\
\end{array}
\right).
\]
In addition, any symmetric circulant matrix of size $n$ can be expressed as follows
\[
A = a_0 I + a_1 \Pi + a_2 \Pi^2 + \dots + a_{n-1} \Pi^{n-1},
\]
where $(a_0,a_1, \dots, a_{n-1})$ is an almost symmetric vector and $\Pi = \operatorname{circ}(0,1,0,\dots, 0)$.

The following result gives us information about the solution of a symmetric circulant system when the vector of constant terms is almost symmetric.
\begin{pro}[Circulant matrices]\label{prop1}
	Let $A$ be a $n$-dimensional invertible circulant symmetric matrix and let $\mathbf{b}$ be an almost symmetric vector of dimension $n$, then $\mathbf{x} = A^{-1} \mathbf{b}$, the solution of the linear system $A \mathbf{x} =\mathbf{b}$, is an almost symmetric vector.
\end{pro}
\begin{proof}
	Since $A$ is a invertible matrix, we know that $\mathbf{x}$ is the unique vector of dimension $n$ satisfying $A\mathbf{x} = \mathbf{b}$ and this vector $\mathbf{x}$ is almost symmetric if and only if $\mathbf{x} = J \mathbf{x}$. So, it is sufficient to prove that $J \mathbf{x}$ is also a solution of the linear system, i.e.\ $A (J\mathbf{x}) = \mathbf{b}$.  Since $\mathbf{b}$ is almost symmetric, the equation $A (J \mathbf{x}) = \mathbf{b}$ becomes equivalent to
	\begin{equation}\label{eq:AJAX}
		J A (J \mathbf{x}) = \mathbf{b}.
	\end{equation}
	It is sufficient to prove that $JAJ = A$.
	Note that the matrix $J$ is an involutory matrix, i.e.\ $J^{-1} = J$, and
	\begin{eqnarray*}
		J A J &=& J \big(a_0 I + a_1 \Pi + a_2 \Pi^2 + \dots + a_{n-1} \Pi^{n-1}\big)J\nonumber\\
		&=& a_0 I + a_1 J \Pi J + a_2 J \Pi^2 J + \dots + a_{n-1} J \Pi^{n-1} J. % \label{JAJ_expr}
	\end{eqnarray*}
	
	The matrix $\Pi$ is orthogonal, satisfying $\Pi^{-1} = \Pi^T$. Moreover,
	\begin{eqnarray*}
		J \Pi J &=& J (\Pi J) = J
		\left(
		\begin{array}{cccccc}
			0 & 0 & \dots & 0 & 0 & 1 \\
			0 & 0 & \dots  & 0 & 1 & 0 \\
			0 & 0 & \dots & 1 & 0 & 0 \\
			\vdots & \vdots & \iddots & \iddots & \vdots & \vdots \\
			0 & 1 & 0 & 0 & 0 & 0 \\
			1 & 0 & 0 & 0 & 0 & 0 \\
		\end{array}
		\right) = \Pi^T,
	\end{eqnarray*}
	which implies
	\(
		J \Pi^n J = J \Pi J^2 \Pi^{n-1} J = \Pi^T J \Pi^{n-1} J = \dots = \big(\Pi^T\big)^n.
	\)
	Thus, we get
	\begin{eqnarray*}
		J A J &=& a_0 I + a_1 \Pi^T + a_2 \big(\Pi^T\big)^2 + \dots + a_{n-1} \big(\Pi^T\big)^{n-1}\\
		&=& \big( a_0 I + a_1 \Pi + a_2 (\Pi)^2 + \dots + a_{n-1} (\Pi)^{n-1} \big)^T\\
		&=& A^T = A.
	\end{eqnarray*}
	Thus, (\ref{eq:AJAX}) holds and hence $J\mathbf{x}$ is solution of the equation $A \mathbf{x} = \mathbf{b}$. By uniqueness of the solution we get $\mathbf{x} = J\mathbf{x}$, proving that $\mathbf{x}$ is almost symmetric.
\end{proof}

Because the $K \times K$ matrix $A_K$ in (\ref{eq:SK}) is a symmetric circulant matrix, it is possible to obtain explicit formulas for all its eigenvalues and eigenvectors using \cite[Thm.\ 3.2.2]{davis_circulant}. Since all its eigenvalues are non-null, we conclude that the matrix $A_K$ is invertible. Thus, using Proposition \ref{prop1}, the linear system (\ref{eq:SK}) has as its unique solution the vector $\mathbf{s}_K$, which is almost symmetric. In addition to its almost symmetry, the vector $\mathbf{b}_K$ satisfies $b_1=b_{K-1}$, $b_k = 0,\; 2 \leq k \leq K-2$. This simple structure of $\mathbf{b}_K$ allows us to deduce explicit expressions for $s_k, \; 0 \leq k \leq K-1$, given in Theorem \ref{thm:sk_expl}, which is proved below.

\subsection{Proof of Theorem \ref{thm:sk_expl}}\label{sec:thm:sk_expl}

Consider the four families of orthogonal
polynomials
$N_{\even,n}(x)$, $D_{\even, n}(x)$, $N_{\odd, n}(x)$,
$D_{\odd, n}(x)$,  $n\geq 0$, defined by
\[
\begin{array}{llll}
	N_{\even, 0}(x) = 2, & D_{\even, 0}(x) = 0, & N_{\odd, 0}(x) = 1, & D_{\odd, 0}(x) = 1 \\
	N_{\even, 1}(x) = x, & D_{\even, 1}(x) = x+2, & N_{\odd, 1}(x) = x-1,  & D_{\odd, 1}(x) = x+1, \\
\end{array}
\]
satisfying all of them the recurrence relation
\begin{equation}\label{recurrent_relation}
	p_{n+1}(x) = x \, p_{n}(x) - p_{n-1}(x),
\end{equation}
for all $n \ge 1$.

The next proposition will prove useful in the sequel.

\begin{lem}\label{prop:recurrence}
	The following relations hold, for all $n \geq 0$:
	\begin{eqnarray}
		2 N_{\even,\,n}(x) - x N_{\even,\,n+1}(x) + (x-2) D_{\even,\,n+1}(x) &=& 0,  \label{eq:relEven} \\
		2 N_{\odd,\,n}(x) - x N_{\odd,\,n+1}(x) + (x-2) D_{\odd,\,n+1}(x) &=& 0. \label{eq:relOdd}
	\end{eqnarray}
	Furthermore, we have the following identities involving the Chebyshev polynomials of first and second kind, for all $n \ge 0$:
	\begin{eqnarray}
		N_{\even,\,n}(x) &=& 2 \, T_n(x/2), \label{Neven_cheby}\\
		D_{\even,\,n}(x) &=& (x+2) \, U_{n-1}(x/2), \label{Deven_cheby} \\
		N_{\odd,\,n}(x) &=& U_n(x/2) - U_{n-1}(x/2), \label{Nodd_cheby}\\
		D_{\odd,\,n}(x) &=& U_n(x/2) + U_{n-1}(x/2).\label{Dodd_cheby}
	\end{eqnarray}
\end{lem}

\begin{proof}
	Setting
	$
	P_n(x) =  2 N_{\even, n}(x) - x N_{\even, n+1}(x) + (x-2) D_{\even, n+1}(x)$, for all $n\geq 0$,
	it follows from the definitions of $N_{\even, n}(x)$ and $D_{\even, n}(x)$ that
	$P_0(x) = 0, P_1(x)=0$ and $P_n(x)$ satisfies the recurrence relation (\ref{recurrent_relation}).
	Therefore $P_n(x) = 0$ for every $n \geq 0$ and (\ref{eq:relEven}) is proved.
	The proof of (\ref{eq:relOdd}) is similar.
	
	Now, note that the sequence of polynomials $(2 \, T_n(x/2))_{n \ge 0}$ satisfy the recurrence relation (\ref{recur_cheby}). Furthermore, $2 \, T_0(x/2) = 2 = N_{\even,\,0}(x)$ and $2 \, T_1(x/2) = x = N_{\even,\,1}(x)$. Consequently, identity (\ref{Neven_cheby}) is proved. Analogously, identities (\ref{Deven_cheby}), (\ref{Nodd_cheby}) and (\ref{Dodd_cheby}) are proved.
\end{proof}

We now prove Lemma \ref{lemma:sk_expl}, which provides explicit expressions for $s_k$, $k \in \{0,1,\dots, K-1\}$, in terms of the polynomials $N_{\even,n}(x)$, $D_{\even, n}(x)$, $N_{\odd, n}(x)$ and $D_{\odd, n}(x)$.

\begin{lem}[Explicit formulas for $s_k$] \label{lemma:sk_expl}
	The values of $s_k$, $0 \le k \le K-1$, are given by
	\begin{enumerate}
		\item[a)] If $K= 2 K_2$, $K_2 \geq 2$,
		\begin{eqnarray}
			& & s_0 =  \displaystyle \frac{N-1}{K N} \; \frac{N_{\even, K_2}(\beta_N)  }{D_{\even, K_2}(\beta_N) } + \frac{1}{K N} \label{eq:solEven0}
			, \\
			& & s_k =  \displaystyle \frac{N-1}{K N} \; \frac{N_{\even, K_2-k}(\beta_N)  }{D_{\even, K_2}(\beta_N) } ,
			\; 1 \leq k \leq K_2, \label{eq:solEven} \\
			& & s_{K-k} = s_k, \;  1 \leq k \leq K_2-1,\nonumber
		\end{eqnarray}
		\item[b)] If $K = 2 K_2 + 1$, $K_2 \geq 1$,
		\begin{eqnarray}
			& & s_0 =  \displaystyle \frac{N-1}{K N} \; \frac{N_{\odd, K_2}(\beta_N)  }{D_{\odd, K_2}(\beta_N) } + \frac{1}{K N}
			,\nonumber  \\
			& & s_k =  \displaystyle \frac{N-1}{K N} \; \frac{N_{\odd, K_2-k}(\beta_N)  }{D_{\odd, K_2}(\beta_N) } ,
			\; 1 \leq k \leq K_2, \nonumber  \label{eq:solOdd}  \\
			& & s_{K-k} = s_k, \;  1 \leq k \leq K_2,\nonumber
		\end{eqnarray}
	\end{enumerate}
	where $\beta_N$ is defined by (\ref{eq:params}).
\end{lem}

\begin{proof}
	We separate the proof into two cases: when $K$ is even and when $K$ is odd.
	
	$(a)$ When $K$ is even, say $K= 2 K_2$, Equation (\ref{eq:SK}) is equivalent to the following linear system for $s_k, \; 0 \leq k \leq K_2$,
	\begin{eqnarray}
		\beta_N s_0 - 2 s_1 &=& -\frac{1}{K N} \gamma_N , \label{eq:vers0} \\
		- s_0 + \beta_N s_1 - s_2 &=& -\frac{1}{K N}, \label{eq:vers1} \\
		-s_{k-1} + \beta_N s_{k} - s_{k+1} &=& 0,  \label{eq:versk}
	\end{eqnarray}
	for $2 \leq k \leq K_2-1$ and
	\begin{equation}
		\beta_N s_{K_2} - 2 s_{K_2-1} = 0.  \label{eq:versK2}
	\end{equation}
	
	Note that (\ref{eq:versK2}) follows from the equality $s_{K_2 - 1} = s_{K_2 + 1}$.
	
	Consider $A \in \mathbb{R}$ such that $s_{K_2} = 2 A = N_{\even, 0}(\beta_N) A$.
	Equation (\ref{eq:versK2}) implies
	\begin{equation*}
		s_{K_2-1} = A \beta_N = A N_{\even, 1}(\beta_N).
	\end{equation*}
	Equation (\ref{eq:versk}) may be written as
	\begin{equation}\nonumber
		s_{k-1} = \beta_N s_k - s_{k+1},
	\end{equation}
	for $2 \leq k \leq K_2-1$.
	This proves that $s_{k}$, for $k$ decreasing from $K_2$ to $1$, may be written
	\begin{equation}%\label{sK_even}
		s_k = A N_{\even,K_2-k}(\beta_N). \nonumber
	\end{equation}
	
	From Equation (\ref{eq:vers1}), we get
	\begin{eqnarray}
		s_0 &=& \beta_N s_1 - s_2 + \frac{1}{K N} \nonumber \\
		&=& A \left[ \beta_N  N_{\even,K_2-1}(\beta_N) - N_{\even,K_2-2}(\beta_N) \right] + \frac{1}{K N} \nonumber \\
		&=& A N_{\even,K_2}(\beta_N) + \frac{1}{K N}. \label{eq:s01a}
	\end{eqnarray}
	
	Plugging (\ref{eq:s01a}) into Equation (\ref{eq:vers0}), we get
	\begin{eqnarray}
		A \left[ \beta_N N_{\even,K_2}(\beta_N) -2 N_{\even,K_2-1}(\beta_N) \right] &=&  - \frac{1}{K N}( \beta_N + \gamma_N) \nonumber \\
		&=& \frac{1}{K N} \; \frac{2p}{1+\theta}.\label{eqA1}
	\end{eqnarray}
	Using Equation (\ref{eq:relEven}) we get
	\begin{eqnarray}
		A[ \beta_N N_{\even,K_2}(\beta_N) -2 N_{\even,K_2-1}(\beta_N) ] &=& A (\beta_N - 2) D_{\even,K_2}(\beta_N) \nonumber \\
		&=& A \frac{2p}{(N-1)(1+\theta)} D_{\even,K_2}(\beta_N).\label{eqA2}
	\end{eqnarray}
	Thus, using (\ref{eqA1}) and (\ref{eqA2}), we obtain
	\(
	A = \frac{1}{KN} \frac{N-1}{ D_{\even,K_2}(\beta_N)}, \nonumber
	\)
	that achieves the proof of (\ref{eq:solEven}) for an even value of $K$.
	
	$(b)$ The proof when $K$ is odd is similar. Indeed, for $K = 2 K_2+1$, the linear system for $s_k$, with $0 \leq k \leq K_2$, is
	\begin{eqnarray}
		\beta_N s_0 - 2 s_1 &=& -\frac{1}{K N} \gamma_N , \label{eq:vers0_b} \\
		- s_0 + \beta_N s_1 - s_2 &=& -\frac{1}{K N}, \label{eq:vers1_b} \\
		-s_{k-1} + \beta_N s_{k} - s_{k+1} &=& 0,\label{eq:versk_b}
	\end{eqnarray}
	for $2 \leq k \leq K_2-1$ and
	\begin{equation}
		-   s_{K_2}  + \beta_N s_{K_2} - s_{K_2-1} = 0.  \label{eq:versK2_b}
	\end{equation}
	
	Equation (\ref{eq:versK2_b}) may be written as
	$$
	(\beta_N -1) s_{K_2} = s_{K_2-1},
	$$
	and so
	$$
	s_{K_2} = B = B N_{\odd, 0}(\beta_N), \; s_{K_2-1} = B (\beta_N - 1) = B N_{\odd, 1}(\beta_N).
	$$
	From Equations (\ref{eq:relOdd}) and (\ref{eq:versk_b}), it follows that
	\begin{equation}
		s_k = B N_{\odd,K_2-k}(\beta_N), \; 1 \leq k \leq K_2. \nonumber
	\end{equation}
	
	Then, from Equation (\ref{eq:vers1_b}), we get
	\begin{eqnarray}
		s_0 &=& \beta_N s_1 - s_2 + \frac{1}{K N} \nonumber \\
		&=& B \left[ \beta_N  N_{\odd,K_2-1}(\beta_N) - N_{\odd,K_2-2}(\beta_N) \right] + \frac{1}{NK} \nonumber \\
		&=& B N_{\odd,K_2}(\beta_N) + \frac{1}{K N}. \nonumber \label{eq:s01}
	\end{eqnarray}
	From (\ref{eq:vers0_b}), it follows, using (\ref{eq:relOdd}), that
	\(
	B = \frac{N-1}{K N D_{\odd,K_2}(\beta_N)}. \nonumber
	\)
	The proof of Lemma \ref{lemma:sk_expl} is therefore complete.
\end{proof}

We are now able to prove Theorem \ref{thm:sk_expl}, which provides explicit expressions for the covariances of the proportions of particles in two states under the stationary distribution, in terms of the orthogonal Chebyshev polynomials of first and second kind.

\begin{proof}[Proof of Theorem \ref{thm:sk_expl}]
	Using expressions (\ref{Neven_cheby}), (\ref{Deven_cheby}), (\ref{Nodd_cheby}) and (\ref{Dodd_cheby}), and Lemma \ref{lemma:sk_expl} we obtain explicit expressions for $s_k$ in terms of the Chebyshev polynomials of first and second kind, for $0 \leq k \leq K-1$. Since $\operatorname{Cov}_{\nu_N}[\eta(0)/N, \eta(k)/N] = s_k - 1/K^2$, for all $0 \leq k \leq K-1$, we deduce that (\ref{sys1:s0even}), (\ref{sys1:skeven}), (\ref{sys1:s0odd}) and (\ref{sys1:skodd}) hold.
\end{proof}

Now, using Theorem \ref{thm:sk_expl} we are able to study the monotony of the covariance of the proportions of particles in two sites as a function of the graph distances between these two sites.

\begin{proof}[Proof of Corollary \ref{corollary_geometry}]
	
	Note that $\operatorname{Cov}_{\nu_N}\left[ \frac{\eta(0)}{N}, \frac{\eta(k)}{N} \right] \ge \operatorname{Cov}_{\nu_N}\left[ \frac{\eta(0)}{N}, \frac{\eta(k+1)}{N} \right]$ holds  if and only if $s_k \ge s_{k+1}$, for all $k=0,1,\dots, \lfloor \frac{K}{2} \rfloor$.
	So, for $K$ even, using (\ref{sys1:s0even}) and (\ref{sys1:skeven}), it is sufficient to prove that $T_{k+1}(\beta_N/2) \ge T_{k}(\beta_N/2)$. Let us prove it by induction. 
	We know that $ T_1(\beta_N/2) = \beta_N/2 \ge 1 = T_0(\beta_N/2).$
	Assume that $T_{k}(\beta_N/2) \ge T_{k-1}(\beta_N/2)$. Since $\big(T_n(x)\big)_{n \ge 0}$ satisfies the recurrence relation (\ref{recur_cheby}) we have
	\begin{eqnarray*}
		T_{k+1}(\beta_N/2) - T_k(\beta_N/2) = (\beta_N - 1) T_{k}(\beta_N/2) - T_{k-1}(\beta_N/2) \ge T_{k}(\beta_N/2) - T_{k-1}(\beta_N/2) \ge 0,
	\end{eqnarray*}
	where the first inequality is due to the inequality $\beta_N \ge 2$ and the second one because, by assumption,  $T_{k}(\beta_N/2) \ge T_{k-1}(\beta_N/2)$. Then, $T_{k+1}(\beta_N/2) \ge T_{k}(\beta_N/2)$, for all $k \ge 0$.
	
	Analogously, for $K$ odd the inequality $\operatorname{Cov}_{\nu_N}\left[ \frac{\eta(0)}{N}, \frac{\eta(k)}{N} \right] \ge \operatorname{Cov}_{\nu_N}\left[ \frac{\eta(0)}{N}, \frac{\eta(k+1)}{N} \right]$ holds for all $k = 0,1,\dots, \lfloor \frac{K}{2} \rfloor$ if
	\begin{equation}\label{condit_Uk}
		U_{k+1}(\beta_N/2) - U_k(\beta_N/2) \ge U_k(\beta_N/2) - U_{k-1}(\beta_N/2),
	\end{equation}
	for all $k \ge 1$. For $k = 1$ we have that (\ref{condit_Uk}) is equivalent to \( \beta_N^2 - 2 \beta_N \ge 0 \),
	which is trivially true since $\beta_N \ge 2$.
	Assume that (\ref{condit_Uk}) holds and let us prove the inequality for $k+1$.
	Indeed, using that $\big(U_n\big)_{n \ge 0}$ satisfies the recurrence relation (\ref{recur_cheby}), we have
	\[
	U_{k+2}(\beta_N/2) - U_{k+1}(\beta_N/2) = (\beta_N - 1) U_{k+1}(\beta_N/2) - U_{k}(\beta_N/2) \ge U_{k+1}(\beta_N/2) - U_{k}(\beta_N/2).
	\]
	Thus, (\ref{condit_Uk}) holds for all $k = 0,1, \dots, K_2$.
\end{proof}

\subsection{Proof of Theorem \ref{thm_two_particles_covariance}}\label{subsec_thm_assymptotic}

Theorem \ref{thm:sk_expl} allows us to get a Taylor series expansion  for
$s_k$, $0 \leq k \leq K-1$ as a function of $\frac{1}{N}$, as soon as we are able to obtain such a series expansion for $\beta_N$, as a function of $1/N$, as well as for the polynomials
$N_{\odd, n}(x)$, $N_{\even,n}(x)$,
$D_{\odd, n}(x)$, $D_{\mbox{{\tiny even}}, n}(x)$, $n\geq 0$
around $x = 2$, using their definitions by induction given in (\ref{recurrent_relation}).

\begin{lem}\label{lema:taylor_expressions}
	The polynomials $N_{\odd, n}(x)$, $N_{\even,n}(x)$, $D_{\odd, n}(x)$, $D_{\mbox{{\tiny even}}, n}(x)$, for $n\geq 0$, satisfy the following Taylor series expansion of order $2$ around $x = 2$:
	\begin{eqnarray}
		N_{\even,n}(x) &=& 2 + n^2 (x-2) + \frac{n^4 - n^2}{12} (x - 2)^2 + o(x-2)^2, \label{eq:NevenDL} \\
		D_{\even,n}(x) &=& 4n + \frac{2n^3 + n}{3} (x - 2)+ \frac{n^5 - n}{30} (x-2)^2 + o(x-2)^2, \label{eq:DevenDL}\\
		N_{\odd,n}(x) &=& 1 + \frac{n^2 + n}{2}(x - 2) + \frac{n^4 + 2 n^3 - n^2 - 2n}{24} (x - 2)^2 + o(x-2)^2, \label{eq:NoddDL} \\
		D_{\odd,n}(x) &=& 2n+1 + \frac{2n^3 + 3n^2 + n}{6}(x - 2) +  \frac{2n^5 + 5n^4 - 5n^2 - 2n}{120}(x - 2)^2 + \nonumber\\
		& & + \; o(x -2)^2.\label{eq:DoddDL}
	\end{eqnarray}
\end{lem}

\begin{proof}
	Assume $N_{\even,n}(x) = a_0^{(n)} + a_1^{(n)}(x - 2) + a_2^{(n)} (x-2)^2 + o(x-2)^2$, for all $n \ge 0$. Note that the polynomials $N_{\even,n}(x)$ can also be defined as
	\begin{eqnarray}
		N_{\even, 0}(x) &=& 2, \nonumber \\
		N_{\even, 1}(x) &=& (x-2) + 2, \nonumber \\
		N_{\even, n}(x) &=& (x - 2) N_{\even, n-1}(x) + 2 N_{\even, n-1}(x) - N_{\even, n-2}(x), \; n \geq 2.\label{recurrence_adapted}
	\end{eqnarray}
	Thus, the coefficients $\big(a_0^{(n)}\big)_{n \ge 0}$ satisfy the recurrence relation $a_0^{(0)} = a_0^{(1)} = 2$ and $a_0^{(n)} = 2 a_0^{(n-1)} -  a_0^{(n-2)}$, for every $n \ge 2$, which yields $a_0^{(n)} = 2$, for all $n \ge 0$.
	
	Also, using (\ref{recurrence_adapted}), the coefficients $\big(a_1^{(n)}\big)_{n \ge 0}$ satisfy $a_1^{(0)} = 0$, $a_1^{(1)} = 1$ and
	\[
	a_1^{(n)} = 2 a_1^{(n-1)} -  a_1^{(n-2)} + a_0^{(n-1)} = 2 a_1^{(n-1)} -  a_1^{(n-2)} + 2,
	\]
	for all $n \ge 0$.
	Solving this recurrence gives $a_1^{(n)} = n^2$, for all $n \ge 2$.
	
	Similarly, the coefficients $\big(a_2^{(n)}\big)_{n \ge 0}$ satisfy $a_2^{(0)} = a_2^{(1)} = 0$ and
	\[
	a_2^{(n)} = 2 a_2^{(n-1)} -  a_2^{(n-2)} + a_1^{(n-1)} = 2 a_2^{(n-1)} -  a_2^{(n-2)} + (n-1)^2,
	\]
	for all $n \ge 0$.
	which yields $\displaystyle a_2^{(n)} = \frac{n^4 - n^2}{12}$, for all $n \ge 2$, proving (\ref{eq:NevenDL}).
	
	The proofs of (\ref{eq:DevenDL}), (\ref{eq:NoddDL}) and (\ref{eq:DoddDL}) are similar.
\end{proof}

We now prove Theorem \ref{thm_two_particles_covariance}, which provides a second order Taylor series expansion of the variance of the proportion of particles in each state, as a function of $1/N$,  when $N$ tends to infinity.

\begin{proof}[Proof of Theorem \ref{thm_two_particles_covariance}]
	
	Suppose $K$ is even, say $K = 2 K_2$. Using Lemma \ref{lemma:sk_expl}, we have
	\[
	s_k  = \frac{1}{K} \left( 1 - \frac{1}{N} \right) \frac{N_{\even, K_2 - k}(\beta_N)}{D_{\even, K_2}(\beta_N)},
	\]
	for all $k = 1,2, \dots, K_2$. Note that $\beta_N$, defined by (\ref{eq:params}), tends to $2$ when $N$ tends to infinity, specifically
	\[
	\beta_N - 2 = \frac{2 p}{(N-1)(1+\theta)} = \frac{2p}{1+\theta} \left( \frac{1}{N} + \frac{1}{N^2}\right) + o\left( \frac{1}{N^2} \right).
	\]
	Using (\ref{eq:NevenDL}) and (\ref{eq:DevenDL}), we have
	\begin{eqnarray}
		\frac{N_{\even, K_2 - k}(\beta_N)}{D_{\even, K_2}(\beta_N)} &=& \frac{2 + (K_2 - k)^2 (\beta_N-2) + \frac{(K_2 - k)^4 - (K_2 - k)^2}{12} (\beta_N - 2)^2 + o\left((\beta_N-2)^2\right)}{4K_2 + \frac{2K_2^3 + K_2}{3} (\beta_N - 2)+ \frac{K_2^5 - K_2}{30} (\beta_N-2)^2 + o\left((\beta_N-2)^2\right)} \nonumber\\
		&=& \frac{1}{K} + \frac{\left(6 k(k-K)+K^2-1\right)}{12 K} (\beta_N - 2) \nonumber \\
		& & + \; \frac{30 k (K - k)[k(K-k) + 2] - (K^2 - 1) (K^2 + 11)}{720 K}(\beta_N - 2)^2 \nonumber \\
		& & + \;  o \left((\beta_N - 2)^2\right), \label{DLfractionEven}
	\end{eqnarray}
	where $K = 2 K_2$.

	Finally,
	\begin{eqnarray*}
		s_k &=& \frac{1}{K^2} + \left( - 1 + \frac{(6 k (k-K)  +K^2-1)}{6 } \frac{p}{1 + \theta}\right) \frac{1}{K^2 N} \\
		& & + \;  \frac{  30 k (K - k)[k(K-k) + 2] -(K^2 - 1) \left(K^2+11\right)}{180}  \left( \frac{p}{1 + \theta} \right)^2  \frac{1}{K^2 N^2} + o \left( \frac{1}{N^2} \right).
	\end{eqnarray*}
	Using (\ref{eq:solEven0}), we get the following expression for $s_0$,
	\[
		s_0  = \frac{1}{K^2} + \left(K - 1 +\frac{K^2-1}{6 } \frac{p}{1 + \theta}\right) \frac{1}{K^2 N}  +   \frac{ \left(K^2-1 \right) \left(K^2+11\right)}{180} \left( \frac{p}{1 + \theta} \right)^2  \frac{1}{K^2 N^2} + o \left( \frac{1}{N} \right).
	\]
	Now, the expression (\ref{DLcovariancesK}) for $\operatorname{Cov}_{\nu_N}\left[ {\eta(0)}/{N}, {\eta(k)}/{N} \right]$ with $K$ even follows by noting that $\mathbb{E}_{\nu_N}\left[ \frac{\eta(k)}{N} \right] = \frac{1}{K}$, for all $k = 0,1,2, \dots, K-1$.
	
	Considering $K$ odd, specifically $K = 2 K_2 + 1$, and using (\ref{eq:NoddDL}) and (\ref{eq:DoddDL}), we have
	\begin{eqnarray*}
		\frac{N_{\odd, K_2 - k}(\beta_N)}{D_{\odd, K_2}(\beta_N)}
		&=& \frac{1}{K}+\frac{\left(6 k( k - K) + K^2-1 \right)}{12 K}(\beta_N - 2) \\
		& & + \; \frac{30 k (K - k)[k(K-k) + 2] - (K^2 - 1) (K^2 + 11)}{720 K}(\beta_N - 2)^2\\
		& & + \; o\left((\beta_N -2)^2\right),
	\end{eqnarray*}
	which is the same expression we get for $\displaystyle \frac{N_{\even, K_2 - k}(\beta_N)}{D_{\even, K_2}(\beta_N)}$ in (\ref{DLfractionEven}). So, the general result is proved.
\end{proof}

\begin{proof}[Proof of Corollary \ref{corol:Convergence_QSD}]
	Using Jensen's inequality, we have
	\begin{eqnarray}
		\mathbb{ E }_{ \nu_N }\left[ \left\| m ( \eta ) - \nu_{\mathrm{qs}} \right\|_{ 2 } \right] &\leq& \left(\mathbb{E}_{\nu_N} \left\| m(\eta) - \nu_{\mathrm{qs}} \right\|_{2}^{2} \right)^{1/2} \nonumber \\
		&=& \left(\sum_{k=0}^{ K-1 } \operatorname{Var}_{\nu_N} \left[ \frac{\eta(k)}{N} \right] \right)^{1/2} \nonumber \\
		&=& \sqrt{ K } \left( \operatorname{Var}_{\nu_{ N }} \left[ \frac{\eta(0)}{N} \right] \right)^{1/2}. \label{ineq:bounding_norm2}
	\end{eqnarray}
	Finally, (\ref{ineq:bounding_norm2_corol}) is proved using (\ref{ineq:bounding_norm2}) and Theorem \ref{thm_two_particles_covariance}.
\end{proof}

\section{Covariances of the proportions of particles at a given time }\label{sec:ergo}

\subsection{Proof of Theorem \ref{ergodicity_thm}}

\begin{proof}[Proof of Theorem \ref{ergodicity_thm}]
	
	Consider $\eta \in \mathcal{E}_{K,N}$ and the function $f_k: \eta \mapsto \eta(k)$, for $k \in \{0,1,\dots, K-1\}$. Using the expression of $\mathcal{L}_{K,N} f_k$, for $k=0,1, \dots, K-1$, given by (\ref{exprLBk}), and the Kolmogorov equation, we get
	\begin{eqnarray}
		\frac{\operatorname{d}}{\operatorname{d} t} \mathbb{E}_{\eta} \left[ \displaystyle \frac{f_k\big(\eta_t^{(N)}\big)}{N} \right] &=& \mathbb{E}_{\eta} \left[ \frac{\mathcal{L}_{K,N} f_k\big(\eta_t^{(N)}\big)}{N} \right] \nonumber \\
		&=& \mathbb{E}_{\eta} \left[ \displaystyle \frac{f_{k-1}\big(\eta_t^{(N)}\big)}{N} \right]  -  (1+\theta) \mathbb{E}_{\eta} \left[ \displaystyle \frac{f_k\big(\eta_t^{(N)}\big)}{N} \right] \nonumber \\
		& & +  \theta \, \mathbb{E}_{\eta} \left[ \displaystyle \frac{f_{k+1}\big(\eta_t^{(N)}\big)}{N} \right],\label{edo_expected}
	\end{eqnarray}
	for $k = 0,1,\dots, K-1$.
	
	Let us define $s_t(k) = \mathbb{E}_{\eta} \left[f_k\big(\eta_t^{(N)}\big)/N \right] =  \mathbb{E}_{\eta} \left[\eta_t^{(N)}(k)/N \right] = \overline{m}\big(\eta^{(N)}_t\big)(k)$, for $k=0,1,\dots, K-1$, and the vector $\mathbf{s}_t = (s_t(0), s_t(1), \dots, s_t(K-1))^T$. Using (\ref{edo_expected}), we get that $\mathbf{s}_t$ satisfies the differential equation
	\[
	\frac{ \operatorname{d} \mathbf{s}_t}{\operatorname{d} t} = \mathbf{s}_t Q,
	\]
	where $Q$ is the circulant infinitesimal rate matrix defined in (\ref{Qmatrix_def}),
	with initial condition $\mathbf{s}_0 = \eta/N$.
	Note that the solution of this differential equation is given by
	\[
	\mathbf{s}_t = \frac{\eta}{N} \mathrm{e}^{ t Q }.
	\]
	Thus, $\overline{m}\big(\eta^{(N)}_t\big)$ is actually equal to the distribution of the asymmetric random walk on the cycle graph $\mathbb{Z}/K \mathbb{Z}$ with infinitesimal generator matrix $Q$ and initial distribution $m(\eta)$ at time $t = 0$, which is $\mathcal{L}_{m(\eta)}\left(Z_t \mid t < \tau_p\right)$. 
	So, the proof of formula (\ref{bound_ergo_L2}) follows from (\ref{exp_bound_norm_2distrib}) in Theorem \ref{thm1}.
	
\end{proof}

\subsection{Proof of Theorem \ref{ergodicity_thm_2}}

In order to study the convergence of the empirical distribution $m\big(\eta_t^{(N)}\big)$ induced by the $N$-particle system, we will analyze the behavior of the covariance functions in time. Let $\eta \in \mathcal{E}_{K,N}$ be fixed and let us define the functions $s_t^{(2)}(k,r)$ as $s_t^{(2)}(k,r) = \mathbb{E}_{\eta} \left[ {f(k,r)}/{N^2} \right] = \mathbb{E}_{\eta} \left[ \eta(k)\eta(r)/{N^2} \right]$, for all $k,r \in \mathbb{Z}/K \mathbb{Z}$. Using (\ref{exprLBkk}), (\ref{exprLBkl}) and (\ref{exprLBkl1}), we have
\begin{eqnarray*}
	\frac{\operatorname{d} s_{t}^{(2)}(k,k)}{\operatorname{d} t} &=& 2 \left[s_t^{(2)}(k, k-1) - \left( 1 + \theta + \frac{p}{N-1} \right) s_t^{(2)}(k,k) + \theta  s_t^{(2)}(k, k+1)  \right]\\
	& & + \frac{1}{N} \left[ s_t(k-1) + \left(1 + \theta + 2 \frac{p}{N-1} \right) s_t(k) + \theta s_t(k+1) \right],\\	
	\frac{\operatorname{d} s_{t}^{(2)}(k,k+1)}{\operatorname{d} t} &=& -2 \left( 1 + \theta + \frac{p}{N-1} \right) s_{t}^{(2)}(k,k+1) + s_{t}^{(2)}(k-1,k+1) + \theta s_{t}^{(2)}(k+1,k+1)\\
	& & + s_{t}^{(2)}(k,k) + \theta s_{t}^{(2)}(k,k+2) - \frac{1}{N} \left[ s_{t}(k) + \theta  s_{t}(k+1) \right] \\
	\frac{\operatorname{d} s_{t}^{(2)}(k,k+l)}{\operatorname{d} t} &=& -2 \left( 1 + \theta + \frac{p}{N-1} \right) s_{t}^{(2)}(k,k+l) + s_{t}^{(2)}(k-1,k+l) + \theta s_{t}^{(2)}(k+1,k+l)\\
	& & + s_{t}^{(2)}(k,k+l-1) + \theta s_{t}^{(2)}(k,k+l+1).\\
\end{eqnarray*}

Consider the functions $g_t(k,r)$ defined as
\begin{equation*}
	g_t(k,r) = \operatorname{Cov}_{\eta}\left[ \frac{\eta_t(k)}{N},\, \frac{\eta_t(r)}{N} \right] = s_t^{(2)}(k,r) - s_t(k) s_t(r),
\end{equation*}
for all $k,r \in \mathbb{Z}/K \mathbb{Z}$.

Then, we obtain the following system of differential equations
\begin{eqnarray*}
	\frac{\operatorname{d} g_t(k,k)}{\operatorname{d} t} 
	&=& 2 \left[g_t(k, k-1) - \left( 1 + \theta + \frac{p}{N-1} \right) g_t(k,k) + \theta  g_t(k, k+1)  \right] \nonumber \\
	& & + \frac{1}{N} \left[ s_t(k-1) + \left(1 + \theta + 2 \frac{p}{N-1} \right) s_t(k) + \theta s_t(k+1) \right] - \frac{2 p}{N-1} s_t(k)^2, \label{EDO_cov1}\\
	\frac{\operatorname{d} g_t(k,k+1)}{\operatorname{d} t}
	&=& -2 \left( 1 + \theta + \frac{p}{N-1} \right) g_t(k,k+1) + g_t(k-1,k+1) + \theta g_t(k+1,k+1) \nonumber \\
	& & + g_t(k,k) + \theta g_t(k,k+2) - \frac{1}{N}[s_{t}(k) + \theta  s_{t}(k+1)] - \frac{2p}{N-1} s_t(k) s_t(k+1),\label{EDO_cov2}\\
	\frac{\operatorname{d} g_t(k,l)}{\operatorname{d} t} 
	&=&-2 \left( 1 + \theta + \frac{p}{N-1} \right) g_t(k,l) + g_t(k-1,l) + \theta g_t(k+1,l) \nonumber \\
	& & + g_t(k,l-1) + \theta g_t(k,l+1)  - \frac{2p}{N-1} s_t(k) s_t(l). \nonumber
\end{eqnarray*}

Then, the $K^2$-dimensional vector $\mathbf{g}_t = \big({g}_t(k,r)\big)_{k,r}$ satisfies the differential equation
\begin{eqnarray}\label{EDO_covariances}
	\frac{\operatorname{d} \mathbf{g}_{t}}{\operatorname{d} t} &=& \mathbf{g}_t Q^{(2)}_p + \mathbf{w}_t, \label{EDOQ2dim_for_g_t}
\end{eqnarray}
where $Q^{(2)}_p = Q^{(2)} - 2 \frac{p}{N-1} I$, $I$ is the $K^2$-dimensional identity matrix, the matrix $Q^{(2)} \in M_{\mathbb{{R}}}(K^2)$ is defined as
\begin{equation}\label{definition_Q2}
	Q^{(2)}_{(u,v),(k,r)} =
	\left\{
	\begin{array}{ccl}
		1 & \text{ if } & (k = u+1 \land r = v) \lor (k=u \land r = v+1),\\
		\theta & \text{ if } & (k = u-1 \land r = v) \lor (k=u \land r = v-1),\\
		-2(1 + \theta) & \text{ if } & (k = u) \land (r = v).
	\end{array}
	\right.
\end{equation}
and $\mathbf{w}_t = (w_t(k,r))_{k,r}$ is the $K^2$-vector defined by
\[
w_t(k,r) =
\left\{
\begin{array}{lll}
	\displaystyle \frac{1}{N} \big[ s_t(k-1) + \left(1 + \theta + 2 \frac{p}{N-1} \right) s_t(k) + \theta s_t(k+1) \big]
	- \frac{2 p}{N-1} s_t(k)^2 & \text{ if } & r=k \\
	\displaystyle - \frac{1}{N} \left[s_{t}(k \wedge r) + \theta  s_{t}(k \vee r)\right] - \frac{2p}{N-1} s_t(k) s_t(r) & \text{ if } & |k-r| = 1\\
	\displaystyle - \frac{2p}{N-1} s_t(k) s_t(r) & \text{ if } & |k - r| > 1,
\end{array}
\right.
\]
for all $k,r \in \mathbb{Z}/K \mathbb{Z}$.

Note also that
\begin{eqnarray*}
	{g}_0(k,r) &=& 0, \\
	g_{\infty}(k,r) &=& \lim\limits_{t \rightarrow \infty} g_t(k,r) = \operatorname{Cov}_{\nu_N} \left[ \frac{\eta(k)}{N}, \frac{\eta(r)}{N} \right], 
\end{eqnarray*}
and
\begin{equation*}
	w_{\infty}(k,r) = \lim\limits_{t \rightarrow \infty} w_t(k,r) =
	\left\{
	\begin{array}{ccl}
		\frac{2}{K N} \left(1 + \theta + \frac{p}{N-1}\right) - \frac{2p}{K^2 (N-1)} & \text{ if } & k=r,\\
		- \frac{1}{K N} (1 + \theta)  - \frac{2 p}{K^2 (N-1)} & \text{ if } & |k-r| = 1,\\
		- \frac{2 p}{ (N-1)} \frac{1}{K^2} & \text{ if } & |k-r| > 1,\\
	\end{array}
	\right.
\end{equation*}
for all $k,r \in \mathbb{Z}/K \mathbb{Z}$.

Let $A = (a_{r,c})$ and $B = (b_{r,c})$ be two matrices of dimensions $m \times n$ and $w \times q$, respectively. Recall that the Kronecker product of $A$ and $B$, denoted by $A \otimes B$, is the $m w \times n q$ matrix defined as
\[
A \otimes B =
\left(
\begin{array}{cccc}
	a_{0,0} B & a_{0,1} B & \ldots & a_{0, n-1} B \\
	\vdots & \vdots & \ddots  & \vdots \\
	a_{m-1, 0} B & a_{m-1, 1} B & \ldots & a_{m-1, n-1} B
\end{array}
\right).
\]
It is convenient to index the elements of $A \otimes B$ with two $2$-dimensional index in the following way
$$(A \otimes B)_{(r_1,r_2),(c_1,c_2)} = (A \otimes B)_{r_1 m + r_2, c_1 n + c_2} = a_{r_1, c_1} \, b_{r_2, c_2},$$
for all $0 \le r_1 \le m-1, 0 \le r_2 \le w-1, 0 \le c_1 \le n - 1, 0 \le c_2 \le q - 1$. Now, consider that $m = n$ and $w = q$, i.e.\ $A$ and $B$ are square matrices of dimension $n$ and $q$, respectively. The Kronecker sum of $A$ and $B$, denoted by $A \oplus B$, is defined as $A \oplus B = A \otimes I_q + I_n \otimes B$, where $I_q$ and $I_n$ are the identity matrices of dimension $q$ and $n$, respectively. It is well known that the exponential of matrices transforms Kronecker sums in Kronecker products as follows
\begin{equation}\label{eq:Kronecker_sum_prod}
	\mathrm{e}^{A \oplus B} = \mathrm{e}^A \otimes \mathrm{e}^B.
\end{equation}

See e.g.\ Chapter XIV of \cite{1965KroneckerSum} and \cite{davis_circulant} for the proofs of these results and more details about the Kronecker product and sum of matrices.

\begin{lem}\label{lemma:propertiesQ2}
	The following properties hold:
	\begin{enumerate}
		\item $Q^{(2)} = Q \oplus Q$, \label{eq:oplus}
		\item $\mathrm{e}^{t Q^{(2)}} = \mathrm{e}^{t Q} \otimes \mathrm{e}^{t Q}$. \label{eq:expKronecker}
	\end{enumerate}
	Consequently, the matrix $Q^{(2)}$ is the infinitesimal rate matrix of the independent coupling of two processes driven by the infinitesimal generator matrix $Q$.
\end{lem}

\begin{proof}[Proof of Lemma \ref{lemma:propertiesQ2}]
	Note that using (\ref{definition_Q2}) for all $r_1, r_2, c_1, c_2 \in \{0,1,\dots,K-1\}$, we have
	\[ Q^{(2)}_{(r_1,r_2),(c_1,c_2)} = Q_{r_1,c_1} I_{r_2,c_2} + I_{r_1,c_1} Q_{r_2,c_2} = (Q \oplus Q)_{(r_1,r_2),(c_1,c_2)},\]
	where $I$ is the $K$-dimensional identity matrix. Then, property \ref{eq:oplus} holds.
	Also, using (\ref{eq:Kronecker_sum_prod}) we can easily prove the property \ref{eq:expKronecker}.
	
	All the non-diagonal entries of matrix $Q^{(2)}$ are positive and the sum of each row is null, thus $Q^{(2)}$ is an infinitesimal matrix. Furthermore,
	\[
	\mathrm{e}^{t Q^{(2)}}_{(r_1,r_2),(c_1,c_2)} = \mathrm{e}^{t Q}_{r_1,c_1} \mathrm{e}^{t Q}_{r_2,c_2},
	\]
	which means that $Q^{(2)}$ is the infinitesimal rate matrix of the independent coupling of two processes driven by $Q$.
\end{proof}

Note also that, when $t$ goes to infinity in (\ref{EDOQ2dim_for_g_t}), we get $\mathbf{g}_{\infty} Q_p^{(2)} + \mathbf{w}_{\infty} = 0$. Since $Q^{(2)}$ is the infinitesimal matrix generator of a Markov process and $Q^{(2)}_p = Q^{(2)} - p_N I$, where $p_N = \frac{2 p}{N-1}$, all the eigenvalues of $Q^{(2)}_p$ are strictly negative and thus, $Q^{(2)}_p$ is invertible. Then,
\begin{equation}\label{charact_Minfty}
	\mathbf{g}_{\infty} = - \mathbf{w}_{\infty} \left( Q_p^{(2)} \right)^{-1}.
\end{equation}

We will now prove Theorem \ref{ergodicity_thm_2}, which gives us the solution of the system of differential equations (\ref{EDO_covariances}) and studies the convergence of the proportion of particles at time $t$ in each state when $t$ and $N$ tend to infinity.

\begin{proof}[Proof of Theorem \ref{ergodicity_thm_2}]
	The solutions of the system of differential equations (\ref{EDO_covariances}) is given by
	\begin{eqnarray*}
		\mathbf{g}_t &=& \left( \int_0^t  \mathbf{w}_u \mathrm{e}^{-u Q^{(2)}_p} \mathrm{d} u \right) \mathrm{e}^{t Q^{(2)}_p}\\
		&=& \left( \int_0^t  \big(\mathbf{w}_u - \mathbf{w}_{\infty}\big) \mathrm{e}^{-u Q^{(2)}_p} \mathrm{d} u + \mathbf{w}_{\infty} \int_0^t  \mathrm{e}^{-u Q^{(2)}_p} \mathrm{d} u \right) \mathrm{e}^{t Q^{(2)}_p}\\
		&=& \left( \int_0^t  \big(\mathbf{w}_u - \mathbf{w}_{\infty}\big) \mathrm{e}^{-u Q^{(2)}_p} \mathrm{d} u + \mathbf{w}_{\infty} \left(Q^{(2)}_p\right)^{-1} \left( I - \mathrm{e}^{-t Q^{(2)}_p} \right)  \right) \mathrm{e}^{t Q^{(2)}_p}\\
		&=& \int_0^t  \big(\mathbf{w}_u - \mathbf{w}_{\infty}\big) \mathrm{e}^{(t-u) Q^{(2)}_p} \mathrm{d} u +  \mathbf{g}_{\infty} \left(I - \mathrm{e}^{t Q^{(2)}_p}\right).
	\end{eqnarray*}
	Note that the last equality comes from (\ref{charact_Minfty}). Therefore, we have
	\begin{eqnarray}
		\left\|\mathbf{g}_t - \mathbf{g}_{\infty}  \right\|_{\infty} & \le & \left\| \int_0^t  \big(\mathbf{w}_u - \mathbf{w}_{\infty} \big) \mathrm{e}^{(t-u) Q^{(2)}_p} \mathrm{d} u \right\|_{\infty} + \left\|\mathbf{g}_{\infty}  \left(\mathrm{e}^{t Q^{(2)}_p}\right) \right\|_{\infty} \nonumber \\
		& \le & \int_0^t  \left\| \mathbf{w}_u - \mathbf{w}_{\infty} \right\|_{\infty} \left\| \mathrm{e}^{(t - u) Q^{(2)}_p}  \right\|_{\infty} \mathrm{d} u + \left\| \mathbf{g}_{\infty} \right\|_{\infty} \left\| \mathrm{e}^{t Q^{(2)}_p} \right\|_{\infty} \label{bound_cova}
	\end{eqnarray}
	We get
	\begin{eqnarray}
		\left\| \mathrm{e}^{s Q^{(2)}_p}  \right\|_{\infty} &=& \mathrm{e}^{- p_N s} \left\| \mathrm{e}^{s Q^{(2)}}  \right\|_{\infty} = \mathrm{e}^{- p_N s}, \label{eq:bound_Var_1}
	\end{eqnarray}
	for all $s \ge 0$, where $p_N = \frac{2 p}{N-1}$.
	Note that the second equality in (\ref{eq:bound_Var_1}) comes from the fact that the rows of $\mathrm{e}^{s Q^{(2)}}$ has sum equal to one, for all $s \ge 0$. 
	Using Corollary \ref{corollary_geometry}, or the Cauchy\,--\,Schwarz inequality, we get
	\begin{equation}\label{eq:bound_Var_3}
		\left\| \mathbf{g}_{\infty} \right\|_{\infty} = \operatorname{Var}_{\nu_N} \left[ \frac{\eta(0)}{N} \right].
	\end{equation}
	Using the inequality (\ref{exp_bound_2norm}) we get
	\[
	\left| s_t(k) - \frac{1}{K} \right| \le \| \mathcal{L}_{m(\eta)}(Z_t \mid t < \tau_p) - \nu_{\mathrm{qs}} \|_2 \le \sqrt{\frac{K - 1}{K}} \mathrm{e}^{- \rho_K t},
	\]
	for every $k \in \mathbb{Z}/K \mathbb{Z}$ and all $t \ge 0$. Therefore,
	\[
		|\mathbf{w}_u(k,k) - \mathbf{w}_{\infty}(k,k)| \le \frac{2}{N}\left(1 + \theta + \frac{p}{N-1}\right) \mathrm{e}^{-\rho_K u} + \frac{2 p}{N-1} \left|s_u(k)^2 - \frac{1}{K^2} \right|. 
	\]
	But
	\[
	\left|s_u(k)^2 - \frac{1}{K^2} \right| = \left(s_u(k) + \frac{1}{K}\right) \left|s_u(k) - \frac{1}{K} \right| \le \frac{K+1}{K}\sqrt{\frac{K-1}{K}} \mathrm{e}^{-\rho_K u}.
	\]
	Thus,
	\begin{eqnarray}
		|\mathbf{w}_u(k,k) - \mathbf{w}_{\infty}(k,k)| &\le& \frac{2}{N}\left(1 + \theta + \frac{p}{N-1} +
		\frac{p}{N-1} \frac{ N(K+1)\sqrt{K-1}}{K\sqrt{K}} \right) \mathrm{e}^{-\rho_K u}. \label{bound_residual_kk}
	\end{eqnarray}
	Similarly we get,
	\begin{eqnarray}
		|\mathbf{w}_u(k,k+1) - \mathbf{w}_{\infty}(k,k+1)| &\le& \frac{2}{N}\left(1 + \theta + \frac{p}{N - 1} \frac{ N(K+1)\sqrt{K-1}}{K\sqrt{K}} \right) \mathrm{e}^{-\rho_K u}, \label{bound_residual_kk1}\\
		|\mathbf{w}_u(k,l) - \mathbf{w}_{\infty}(k,l)| &\le& \frac{2p}{N - 1} \frac{ (K+1)\sqrt{K-1}}{K\sqrt{K}} \mathrm{e}^{-\rho_K u}, \;\; |k - l| \ge 2. \label{bound_residual_kl}
	\end{eqnarray}
	Inequalities (\ref{bound_residual_kk}), (\ref{bound_residual_kk1}) and (\ref{bound_residual_kl}) imply that
	\begin{equation} \label{bound_residual}
		\left\| \mathbf{w}_u - \mathbf{w}_{\infty} \right\|_{\infty} \le C_{K,N} \,\mathrm{e}^{-\rho_K u},
	\end{equation}
	where $C_{K,N}$ is defined by (\ref{def:CKN}).
	Plugging (\ref{eq:bound_Var_1}), (\ref{eq:bound_Var_3}) and (\ref{bound_residual}) into (\ref{bound_cova}), we obtain
	\begin{eqnarray}
		\left\|\mathbf{g}_t - \mathbf{g}_{\infty}\right\|_{\infty} & \le & C_{K,N} \int_0^t  \mathrm{e}^{- \rho_K u} \mathrm{e}^{- p_N (t - u)} \mathrm{d} u +  \mathrm{e}^{- p_N t} \|\mathbf{g}_{\infty}\|_{\infty}   \nonumber \\
		&=& C_{K,N} \mathrm{e}^{- p_N t}\int_0^t  \mathrm{e}^{- (\rho_K - p_N) u} \mathrm{d} u +  \mathrm{e}^{- p_N t} \operatorname{Var}_{\nu_N} \left[ \frac{\eta(0)}{N} \right] \nonumber \\
		&=& C_{K,N} \frac{\mathrm{e}^{-p_N t} - \mathrm{e}^{- \rho_K t}}{\rho_K - p_N} + \mathrm{e}^{- p_N t} \operatorname{Var}_{\nu_N} \left[ \frac{\eta(0)}{N} \right] \label{ineq:bounding_var} \\
		&=& C_{K,N} \frac{1 - \mathrm{e}^{- \rho_K t}}{\rho_K} + \operatorname{Var}_{\nu_N} \left[ \frac{\eta(0)}{N} \right] + o\left( \frac{1}{N} \right) \nonumber \\
		&=& \frac{1}{N} \left\{ D_K \frac{1 - \mathrm{e}^{- \rho_K t}}{\rho_K} + E_K  \right\}  + o\left( \frac{1}{N} \right), \nonumber
	\end{eqnarray}
	where $D_K$ and $E_K$ are given by (\ref{def:DEK}). Note that (\ref{thm_var_1}) is obtained from (\ref{ineq:bounding_var}).
	
	In order to prove (\ref{thm_var_2}), note that for every initial distribution $\mu$ in $\mathbb{Z}/K \mathbb{Z}$ and any initial configuration $\eta \in \mathcal{E}_{K,N}$, we get
	\begin{eqnarray}
		\left\|\overline{m}\left(\eta_{t}\right) - \mathcal{L}_{\mu} (Z_t \mid t \le \tau_p) \right\|_2 &\le & 
		\mathbb{E}_{\eta} \big[\left\|m\left(\eta_{t}\right) - \mathcal{L}_{\mu} (Z_t \mid t \le \tau_p) \right\|_2\big] \label{ineq:triangular2_left} \\ 
		&\le& \mathbb{E}_{\eta} \big[\left\|m\left(\eta_{t}\right) - \overline{m}\left(\eta_{t}\right) \right\|_2\big]
		+ \left\|\overline{m}\left(\eta_{t}\right) - \mathcal{L}_{\mu} (Z_t \mid t \le \tau_p) \right\|_2. \label{ineq:triangular2_right}
	\end{eqnarray}
	Inequality (\ref{ineq:triangular2_left}) is obtained using the convexity of the $2$-norm and Jensen's inequality. Inequality (\ref{ineq:triangular2_right}) is proved using the triangular inequality.
	From Theorem \ref{ergodicity_thm} we know that for any initial configuration $\eta \in \mathcal{E}_{K,N}$, we obtain
	\begin{equation}\label{ineq:triangular_empirical2}
		\mathrm{e}^{- \rho_K t} \left| \varphi_{m(\eta)}\left( \frac{2 \pi}{K} \right) - \varphi_{\mu}\left( \frac{2 \pi}{K} \right) \right| \le \left\|\overline{m}\left(\eta_{t}\right) - \mathcal{L}_{\mu} (Z_t \mid t \le \tau_p) \right\|_2  \le \mathrm{e}^{- \rho_K t} \left\| m(\eta) - \mu \right\|_2,
	\end{equation}
	where $\rho_K$ is given by (\ref{def:rhoK}).
	Also,
	\begin{eqnarray}
		\mathbb{E}_{\eta} \big[\left\|m\left(\eta_{t}\right) - \overline{m}\left(\eta_{t}\right) \right\|_2^2\big] &=& \sum_{k = 0}^{K-1} \operatorname{Var}_{\eta} \left[ \frac{\eta_t(k)}{N} \right] \le  K \left\| g_t \right\|_{\infty} \nonumber \\
		&\le& \frac{2 K}{N} \left( D_K\frac{1 - \mathrm{e}^{- \rho_K t}}{\rho_K}  + E_K \right) + o\left( \frac{1}{N} \right), \label{ineq:norm2squared}
	\end{eqnarray}
	where $D_K$ and $E_K$ are defined by (\ref{def:DEK}). Finally, (\ref{thm_var_2}) is proved using (\ref{ineq:triangular2_left}), (\ref{ineq:triangular2_right}), (\ref{ineq:triangular_empirical2}), (\ref{ineq:norm2squared}) and Jensen's  inequality.
\end{proof}

\appendix

\section{Proof of Lemma \ref{teo_dynam}}\label{app:generator}

In order to calculate $\mathcal{L}_{K,N} f_{k}$, note that
\begin{eqnarray*}
	(\mathcal{L}_{K,N} f_{k})(\eta) &=& \sum\limits_{i, j} \eta(i)  \left( \mathbbm{1}_{\{ j = i + 1\}} + \theta \mathbbm{1}_{\{ j = i - 1\}}  + \eta(j) \frac{p}{N-1} \right)\left[ f_{k}\big(T_{i \rightarrow j} \eta\big) - f_{k}\big(\eta\big) \right].
\end{eqnarray*}
But $f_{k}\big(T_{i \rightarrow j} \eta\big) = f_{k}\big(\eta\big)$ if $i \neq k$ and $j \neq k$. Thus,
\begin{eqnarray*}
	(\mathcal{L}_{K,N} f_{k})(\eta) &=& \eta(k) \sum\limits_{j \neq k}   \left( \mathbbm{1}_{\{ j = k + 1\}} + \theta \mathbbm{1}_{\{ j = k - 1\}}  + \eta(j) \frac{p}{N-1} \right) \left[ T_{k \rightarrow j} \eta (k) - \eta(k) \right] \\
	& & + \sum\limits_{i \neq k} \eta(i)  \left( \mathbbm{1}_{\{ k = i + 1\}} + \theta \mathbbm{1}_{\{ k = i - 1\}}  + \eta(k) \frac{p}{N-1} \right) \left[ T_{i \rightarrow k} \eta (k) - \eta(k) \right] \\
	&=& - \eta(k) \left[ 1 + \theta + p \frac{N - \eta(k)}{N-1} \right]  + \eta(k-1) + \theta \eta(k+1) + p \; \eta(k) \frac{N - \eta(k)}{N-1}\\
	&=& \eta(k-1) - (1 + \theta) \eta(k) + \theta \; \eta(k+1),
\end{eqnarray*}
for all $\eta \in \mathcal{E}_{K,N}$. Thus, (\ref{exprLBk}) is proved.

Now, for computing $\mathcal{L}_{K,N} f_{k,l}$ for all $1 \le k,l \le K$, we separate the proof in three cases: $l=k$, $l = k + 1$ and $l > k + 1$, for all $0 \le k \le K-2$.

{\bf{Case} $l=k$}:

From (\ref{generator1}) we have
\[
(\mathcal{L}_{K,N} f_{k,k})(\eta) = \sum\limits_{i, j} \eta(i)  \left( \mathbbm{1}_{\{ j = i + 1\}} + \theta \mathbbm{1}_{\{ j = i - 1\}}  + \eta(j) \frac{p}{N-1} \right) \left[ f_{k,k}\big(T_{i \rightarrow j} \eta\big) - f_{k,k}\big(\eta\big) \right],
\]
for all $\eta \in \mathcal{E}_{K,N}$.
Denote
\[
	S_{i,j}(\eta) = \eta(i) \left( \mathbbm{1}_{\{ j = i + 1\}}  + \theta \mathbbm{1}_{\{ j = i - 1\}}  + \eta(j) \frac{p}{N-1} \right) \left[T_{i \rightarrow j} \eta (k)^2 - \eta(k)^2 \right].
\]
Note that if $\{i,j\} \cap \{k\} = \emptyset$, then we have $S_{i,j}(\eta) = 0$. So,
\[ (\mathcal{L}_{K,N} f_{k,k})(\eta) = \sum_{j \neq k} S_{k, j}(\eta) + \sum_{i \neq k} S_{i,k}(\eta).\]

Note that
\begin{eqnarray}
	\sum_{j \neq k} S_{k,j}(\eta) &=& \sum_{j \neq k } \eta(k) \left( \mathbbm{1}_{\{ j = k + 1\}} + \theta \mathbbm{1}_{\{ j = k - 1\}} + \eta(j) \frac{p}{N-1} \right) \left[ T_{k \rightarrow j} \eta (k)^2 - \eta(k)^2\right] \nonumber \\
	&=&  \eta(k)  \left( 1 + \theta +  \frac{p}{N-1} {\sum_{j \neq k } \eta(j)} \right) \left[ (\eta (k) - 1)^2 - \eta(k)^2\right] \nonumber \\
	&=&    \left( \eta(k) + \theta \; \eta(k) +  p \; \eta(k) \frac{{N - \eta(k)}}{N-1} \right)\left[ - 2 \eta(k) + 1 \right], \label{sum1.1}  \\
	\sum_{i \neq k} S_{i,k}(\eta) &=& \sum_{i \neq k } \eta(i)  \left( \mathbbm{1}_{\{ {k = i + 1}\}} + \theta \mathbbm{1}_{\{ {k = i - 1}\}}  + \eta(k) \frac{p}{N-1} \right) \left[ T_{i \rightarrow k} \eta (k)^2 - \eta(k)^2\right] \nonumber \\
	&=&  \left( {\eta(k - 1)} + \theta \; {\eta(k + 1)}+ \eta(k) \frac{p}{N-1} {\sum_{i \neq k } \eta(i)}  \right) \left[ (\eta (k) + 1)^2 - \eta(k)^2\right] \nonumber \\
	&=&  \left( \eta(k - 1) + \theta \; \eta(k + 1) + p \; \eta(k) \frac{{N - \eta(k)} }{N-1} \right) \left[ 2 \eta(k) + 1 \right]. \label{sum1.2}
\end{eqnarray}

Summing (\ref{sum1.1}) and (\ref{sum1.2}), we obtain
\begin{eqnarray}
	(\mathcal{L}_{K,N} f_{k,k})(\eta) &=& \sum_{j \neq k} S_{k,j}(\eta) + \sum_{i \neq k} S_{i,k}(\eta)\nonumber \\
	&=& 2 \eta(k) \left[ \eta(k - 1) - \eta(k) + \theta (\eta(k+1) - \eta(k) ) \right] \nonumber \\
	& & + (\eta(k) + \eta(k-1)) + \theta (\eta(k+1) + \eta(k)) + 2 p \; \eta(k) \frac{N - \eta(k)}{N-1} \nonumber \\
	&=& 2 \left[\eta(k-1) \eta(k) - \left(1 + \theta + \frac{p}{N-1} \right) \eta(k)^2 + \theta \eta(k) \eta(k+1) \right]  \nonumber \\
	& & + \eta(k-1) + \left(1 + \theta + \frac{2 p N }{N-1} \right) \eta(k) + \theta \eta(k+1),\nonumber
\end{eqnarray}
for all $\eta \in \mathcal{E}_{K,N}$. Thus, (\ref{exprLBkk}) holds.

{\bf{Case} $l = k + 1$}:

From (\ref{generator1}), similarly to the previous case,  we have
\[
	(\mathcal{L}_{K,N} f_{k, k+1}) (\eta) = \sum\limits_{i, j} \eta(i)  \left( \mathbbm{1}_{\{ j = i + 1\}}   + \theta \mathbbm{1}_{\{ j = i - 1\}}   + \eta(j) \frac{p}{N-1} \right) [f_{k,k+1}\left(T_{i \rightarrow j} \eta \right) - f_{k, k+1}(\eta) ].
\]

Denote
\[
R_{i,j} (\eta) =  \eta(i) \left( \mathbbm{1}_{\{ j = i + 1\}} + \theta \mathbbm{1}_{\{ j = i - 1\}} + \eta(j) \frac{p}{N-1} \right)  [ T_{i \rightarrow j} \eta (k) \, T_{i \rightarrow j}\eta(k + 1) - \eta(k) \eta(k+1) ] .
\]
If $\{i, j\} \cap \{k, k+1\} = \emptyset$, then $R_{i,j} = 0$. Thus,
\[
(\mathcal{L}_{K,N} f_{k, k+1})(\eta) = \sum_{j \neq k} R_{k,j}(\eta) + \sum_{i \neq k, k+1} R_{i,k+1}(\eta) + \sum_{j \neq k+1} R_{k+1,j}(\eta) + \sum_{i \neq k, k+1} R_{ i,k }(\eta).
\]

Note that
\begin{eqnarray*}
	\sum_{j \neq k} R_{k,j}(\eta) &=& R_{k,k+1}(\eta) + \sum_{j \neq k, k+1} R_{k,j}(\eta)\\
	&=& \eta(k) [(\eta(k) - 1)(\eta(k+1) + 1) - \eta(k) \eta(k+1)] \left[ 1 + p \frac{\eta(k+1)}{N - 1}\right]\\
	& & + {\sum_{j \neq k, k+1}} \eta(k) [(\eta(k) - 1)\eta(k+1) - \eta(k) \eta(k+1)] \\
	& & \times  \left( \mathbbm{1}_{\{ j = k + 1\}}  + \theta \mathbbm{1}_{\{ {j = k - 1}\}}  + {\eta(j)} \frac{p}{N-1} \right)\\
	&=& {\eta(k)} [\eta(k) -  {\eta(k+1)} - 1] \left[ 1 + p \frac{\eta(k+1)}{N - 1}\right] \\
	& & - {\eta(k) \eta(k+1)} \left( {\theta} + \frac{p}{N-1} {\sum_{j \neq k, k+1} \eta(j) } \right)\\
	&=& \eta(k) [\eta(k) - 1] \left[ 1 + p \frac{\eta(k+1)}{N - 1}\right] \\
	& & - {\eta(k) \eta(k+1)} (1 + \theta) - p \, {\eta(k) \eta(k+1)}\frac{N - \eta(k)}{N-1},\\
	\sum_{i \neq k, k+1} R_{i,k+1} (\eta) &=& {\sum_{i \neq k, k+1} \eta(i)} [\eta (k)(\eta(k + 1) +1) - \eta(k) \eta(k+1) ] \\
	& & \times \left( \mathbbm{1}_{\{ k+1 = i + 1\}} + \theta \mathbbm{1}_{\{ {k+1 = i - 1}\}} + \eta(k+1) \frac{p}{N-1} \right)\\
	&=& \eta(k) \left( \theta {\eta(k+2)}  + p\, \eta(k+1) \frac{{\sum_{i \neq k, k+1} \eta(i)}}{N-1} \right)\\
	&=&  \theta \eta(k) \eta(k+2) + p \, \eta(k) \eta(k+1) \frac{N - \eta(k) - \eta(k+1)}{N-1},\\
	\sum_{j \neq k+1} R_{k+1, j}(\eta) &=& R_{k+1, k}(\eta) + \sum_{j \neq k, k+1} R_{k+1, j}(\eta)\\
	&=& \eta(k+1) [(\eta(k) + 1)(\eta(k+1) - 1) - \eta(k) \eta(k+1)] \left[ \theta  + p \frac{\eta(k)}{N - 1}\right]\\
	& & + {\sum_{j \neq k, k+1}} \eta(k+1) [\eta(k)(\eta(k+1)-1) - \eta(k) \eta(k+1)] \\
	& & \times \left( \mathbbm{1}_{\{ {j = k + 2}\}}  + \theta \mathbbm{1}_{\{ j = k\}}  + p  \frac{{\eta(j)}}{N-1} \right)\\
	&=& {\eta(k+1)} [\eta(k + 1) - {\eta(k)} - 1] \left[ \theta + p \frac{\eta(k)}{N - 1}\right] \\
	& & - {\eta(k) \eta(k+1)}  \left(  {1} +  \frac{p}{N-1} {\sum_{j \neq k, k+1} \eta(j)} \right)\\
	&=& \eta(k+1) [\eta(k + 1) - 1] \left[ \theta + p \frac{\eta(k)}{N - 1}\right] -{ \eta(k) \eta(k+1)} (1 + \theta)\\
	& & - p \,{\eta(k) \eta(k+1)} \frac{N - \eta(k+1)}{N-1}, \\
	\sum_{i \neq k, k+1} R_{ i, k }(\eta) &=& {\sum_{i \neq k, k+1} \eta(i)} [{(\eta (k)+1)\eta(k + 1) - \eta(k) \eta(k+1) }] \\
	& & \times \left( \mathbbm{1}_{\{ {k = i + 1}\}}  + \theta \mathbbm{1}_{\{ k = i - 1\}}  + \eta(k) \frac{p}{N-1} \right)\\
	&=& {\eta(k+1)} \left( {\eta(k-1)}+ p \, \eta(k) \frac{{N - \eta(k) - \eta(k+1)}}{N - 1} \right)\\
	&=& \eta(k-1) \eta(k+1) + p \, \eta(k) \eta(k+1) \frac{N - \eta(k) - \eta(k+1)}{N - 1}. \\
\end{eqnarray*}
Then,
\begin{eqnarray}
	(\mathcal{L}_{K,N} f_{k,k+1})(\eta) &=& -\eta(k) \eta(k+1) \left[ 2 ( 1 + \theta ) + p \frac{2N - \eta(k) - \eta(k+1) - 2 [N - \eta(k) - \eta(k+1)]}{N - 1} \right] \nonumber \\
	& & + \eta(k) [\eta(k) - 1] \left( 1 + p \frac{\eta(k+1)}{N-1} \right) + \eta(k+1) [\eta(k+1) - 1] \left( \theta + p \frac{\eta(k)}{N-1} \right) \nonumber \\
	& & + \eta(k-1) \eta(k+1) + \theta \eta(k) \eta(k+2) \nonumber \\
	&=& -\eta(k) \eta(k+1) \left[ 2  (1 + \theta) + p \frac{{\eta(k)} + {\eta(k+1)}}{N - 1} \right] \nonumber \\
	& & + \eta(k) [\eta(k) - 1] \left( 1 + p \frac{{\eta(k+1)}}{N-1} \right) + \eta(k+1) [\eta(k+1) - 1] \left( \theta + p \frac{{\eta(k)}}{N-1} \right) \nonumber \\
	& & + \eta(k-1) \eta(k+1) + \theta \eta(k) \eta(k+2) \nonumber \\
	&=& -2 \eta(k) \eta(k+1) (1 + \theta) + \eta(k) [\eta(k) - 1]+ \theta \eta(k+1) [\eta(k+1) - 1]  \nonumber \\
	& & - 2 p \frac{\eta(k)\eta(k+1)}{N-1} + \eta(k-1) \eta(k+1) + \theta \eta(k) \eta(k+2) \nonumber \\
	&=& -2 \left( 1 + \theta + \frac{p}{N-1} \right) \eta(k) \eta(k+1) + \eta(k-1) \eta(k+1)  \nonumber \\
	& & + \theta \eta(k+1)^2 + \eta(k)^2 + \theta \eta(k) \eta(k+2) - \eta(k) - \theta \eta(k+1), \nonumber \label{generator_ec2}
\end{eqnarray}
for all $\eta \in \mathcal{E}_{K,N}$, which is equivalent to (\ref{exprLBkl}).

{\bf{Case} $l > k+1$}:

In this case we have
\[
(\mathcal{L}_{K,N} f_{k,l})(\eta) = \sum\limits_{i, j \in F} \eta(i)  \left( \mathbbm{1}_{\{ j = i + 1\}}  + \theta \mathbbm{1}_{\{ j = i - 1\}}  + \eta(j) \frac{p}{N-1} \right) \left[ f_{k,l}\big(T_{i \rightarrow j}\eta \big) - f_{k,l}\big(\eta \big) \right].
\]

Denote
\[
T_{i,j} (\eta) =  \eta(i)  \left( \mathbbm{1}_{\{ j = i + 1\}} + \theta \mathbbm{1}_{\{ j = i - 1\}} + \eta(j) \frac{p}{N-1} \right) [ T_{i \rightarrow j} \eta (k) \, T_{i \rightarrow j}\eta(l) - \eta(k) \eta(l) ].
\]
Obviously, if $\{i,j\} \cap \{k, k+l\} = \emptyset$, then $T_{i,j}(\eta) = 0$.
Thus
\[
(\mathcal{L}_{K,N} f_{k,l}) (\eta) = \sum_{j \neq k} T_{k , j}(\eta) + \sum_{i \neq k, k+l} T_{i , k+l}(\eta) + \sum_{j \neq k+l} T_{k+l , j}(\eta) + \sum_{i \neq k, k+l} T_{ i, k }(\eta).
\]

Note that
\begin{eqnarray*}
	\sum_{j \neq k} T_{k , j}(\eta) &=& T_{k, k+l}(\eta) + \sum_{j \neq k, k+l} T_{k , j}(\eta)\\
	&=& \eta(k) [(\eta(k) - 1) (\eta(k+l) + 1) - \eta(k) \eta(k + l)] p \frac{\eta(k+l)}{N - 1}\\
	& &  + {\sum_{j \neq k, k+l}} \eta(k) [(\eta(k) - 1)\eta(k+l) - \eta(k) \eta(k+l)]  \\
	& & \times \left( \mathbbm{1}_{ \{j = k+1\}}  + \theta \mathbbm{1}_{\{j = k-1\}}  + p \frac{{\eta(j)}}{N-1} \right)\\
	&=& {\eta(k)} [ \eta(k) - \eta(k+l) - 1] p \frac{{\eta(k+l)}}{N - 1} \\
	& & -  {\eta(k) \eta(k+l)} \left( 1 + \theta + \frac{p}{N-1} {\sum_{j \neq k, k+l} \eta(j)} \right)\\
	&=& {\eta(k) \eta(k+l)} \left[ \frac{p}{N - 1} (\eta(k) - 1) - (1 + \theta) - p \frac{N - \eta(k)}{N - 1} \right],\\
	\sum_{i \neq k, k+l} T_{i, k+l}(\eta) &=& {\sum_{i \neq k, k+l} \eta(i)} \left[ \eta(k) (\eta(k+l) + 1) - \eta(k) \eta(k+l) \right] \\
	& & \times \left( \mathbbm{1}_{\{ k+l = i + 1\}} \frac{1}{K} + \mathbbm{1}_{\{ k+l = i - 1\}} \frac{\theta}{K} + \eta(k+l) \frac{p}{N-1} \right)\\
	&=& \eta(k) \left[ \eta(k+l-1) + \theta  \eta(k+l+1) + p \, \eta(k+l) \frac{{N - \eta(k) - \eta(k+l)}}{N - 1}\right]\\
	&=& \eta(k) \eta(k+l-1) + \theta \eta(k) \eta(k+l+1) + p \, \eta(k) \eta(k+l) \frac{{N - \eta(k) - \eta(k+l)}}{N - 1},\\
	\sum_{j \neq k+l} T_{k+l, j}(\eta) &=& T_{k+l, k}(\eta) + \sum_{j \neq k, k+l} T_{k+l, j}(\eta)\\
	&=& \eta(k+l) \left[ (\eta(k) + 1) (\eta(k+l) - 1) - \eta(k) \eta(k + l) \right] p \frac{\eta(k)}{N - 1}\\
	& &  + {\sum_{j \neq k, k+l}} \eta(k + l) [\eta(k) (\eta(k+l)-1) - \eta(k) \eta(k+l)] \\
	& & \times \left( \mathbbm{1}_{ \{j = k+l+1\}} + \theta \mathbbm{1}_{\{j = k+l-1\}} + p \frac{{\eta(j)}}{N-1} \right)\\
	&=&  {\eta(k+l)} \left[ \eta(k+l) - \eta(k) - 1 \right] p \frac{{\eta(k)}}{N - 1} \\
	& & - {\eta(k) \eta(k+l)} \left[ 1 + \theta + \frac{p}{N - 1} {\sum_{j \neq k, k+l} \eta(j)} \right]\\
	&=& \eta(k+l) \left[ \eta(k+l) - \eta(k) - 1 \right] p \frac{\eta(k)}{N - 1}\\
	& & - \eta(k) \eta(k+l) \left[ 1 + \theta + p \frac{{N - \eta(k) - \eta(k + l)}}{N - 1} \right]\\
	&=& \eta(k) \eta(k + l) \left[ (\eta(k+l) - 1) \frac{p}{N - 1} - \frac{1 + \theta}{N} - p \frac{N - \eta(k+l)}{N-1} \right],\\
	\sum_{i \neq k, k+l} T_{i,k}(\eta) &=& {\sum_{i \neq k, k+l} \eta(i)} [(\eta(k) + 1)\eta(k+l) - \eta(k) \eta(k+l)] \\
	& & \times \left( \mathbbm{1}_{\{k = i+1\}}  + \theta \mathbbm{1}_{\{k = i-1\}}  + p \frac{\eta(k)}{N-1} \right)\\
	&=& \eta(k+l) \left[ \eta(k-1) + \theta \eta(k+1) + p \, \eta(k) \frac{{N - \eta(k) - \eta(k+l)}}{N-1} \right]\\
	&=& \eta(k-1)\eta(k+l) + \theta \eta(k+1)\eta(k+l) + p \, \eta(k)\eta(k+l) \frac{{N - \eta(k) - \eta(k+l)}}{N-1}.
\end{eqnarray*}
Thus,
\begin{eqnarray}
	(\mathcal{L}_{K,N} f_{k,l})(\eta) &=& \eta(k) \eta(k+l) \left( \frac{p}{N - 1} [{\eta(k) + \eta(k+l)} - 2]- 2 ( 1 + \theta ) \right. \nonumber \\
	& & \left. - \frac{p}{N - 1} \left[ {2N - \eta(k) - \eta(k+l) - 2[N - \eta(k) - \eta(k+l)]} \right]  \right) \nonumber \\
	& & + \eta(k) \left[ \eta(k + l -1) + \theta \eta(k + l + 1) \right] + \eta(k + l) \left[ \eta(k-1) + \theta \eta(k + 1) \right] \nonumber \\
	&=& - 2\eta(k) \eta(k+l) \left( 1 + \theta + \frac{p}{N - 1}  \right) + \eta(k) \left[ \eta(k + l -1) + \theta \eta(k + l + 1) \right] \nonumber \\
	& & + \eta(k + l) \left[ \eta(k-1) + \theta \eta(k + 1) \right], \nonumber \label{generator_ec3}
\end{eqnarray}
for all $\eta \in \mathcal{E}_{K,N}$, proving (\ref{exprLBkl1}).

\section*{Acknowledgements}

I would like to thank Djalil Chafa\"{\i}, Bertrand Cloez, Simona Grusea and Didier Pinchon for very useful discussions that greatly enriched this work.

\end{document}